\documentclass{amsart}
\usepackage[utf8]{inputenc}
\usepackage{amsmath}
\usepackage{amsthm}
\usepackage{amsfonts}
\usepackage{amssymb}
\usepackage{mathrsfs}
\usepackage{xcolor}
\definecolor{darkgreen}{rgb}{0,0.45,0}
\definecolor{darkred}{rgb}{0.75,0,0}
\definecolor{darkblue}{rgb}{0,0,0.6}
\usepackage[colorlinks,citecolor=darkgreen,linkcolor=darkblue,urlcolor=darkred]{hyperref}
\usepackage{mathtools}
\usepackage{physics}
\usepackage{tikz}
\usepackage{tikz-cd}
\usetikzlibrary{graphs,decorations.pathmorphing,decorations.markings}
\DeclareMathOperator{\ad}{ad}

\usepackage{bbm}
\newcommand{\one}{\mathbbm{1}}
\usepackage[all,cmtip]{xy}

\theoremstyle{plain}
\newtheorem{theorem}{Theorem}[section]
\newtheorem{lemma}[theorem]{Lemma}
\newtheorem{prop}[theorem]{Proposition}
\newtheorem{corollary}[theorem]{Corollary}
\newtheorem{thm}[theorem]{Theorem}
\newtheorem{conj}[theorem]{Conjecture}

\theoremstyle{remark}

\newtheorem{remark}[theorem]{Remark}

\theoremstyle{definition}
\newtheorem{definition}[theorem]{Definition}

\newtheorem{exmp}[theorem]{Example}

\numberwithin{equation}{section}

\newcommand{\Z}{\mathbb Z}
\newcommand{\A}{\mathbb A}
\newcommand{\Q}{\mathbb Q}
\newcommand{\C}{\mathbb C}
\newcommand{\G}{\mathbb G}
\newcommand{\RR}{\mathbb R}

\newcommand{\X}{\mathbb X}
\newcommand{\GL}{\operatorname{GL}}
\let\sl\relax
\newcommand{\sl}{\mathfrak{sl}}

\newcommand{\PU}{\operatorname{PU}}
\newcommand{\U}{\operatorname{U}}
\newcommand{\Hom}{\operatorname{Hom}}
\renewcommand{\O}{\mathcal O}
\newcommand{\g}{\mathfrak g}

\newcommand{\Sym}{\operatorname{Sym}}
\newcommand{\Aut}{\operatorname{Aut}}
\newcommand{\Nm}{\operatorname{Nm}}
\newcommand{\SU}{\operatorname{SU}}
\newcommand{\R}{\operatorname{R}}

\newcommand{\id}{\operatorname{id}}

\newcommand{\coker}{\operatorname{coker}}
\newcommand{\End}{\operatorname{End}}

\newcommand{\Ind}{\operatorname{Ind}}
\newcommand{\disc}{\operatorname{disc}}
\newcommand{\cusp}{\operatorname{cusp}}

\newcommand{\vol}{\operatorname{vol}}
\newcommand{\pr}{\operatorname{pr}}

\newcommand{\ra}{\ensuremath{\displaystyle\mathop{\rightarrow}}}
\newcommand{\la}{\ensuremath{\displaystyle\mathop{\leftarrow}}}
\newcommand{\ang}[1]{\langle #1\rangle}

\newcommand{\set}[1]{\{ #1\}}

\title{Gross's Conjecture: the dihedral case}

\author[P. Baki\'{c}]{Petar Baki\'{c}}
\address{Department of Mathematics\\University of Utah\\155 S 1400 E\\Salt Lake City, UT 84112}
\email{bakic@math.utah.edu}
\author[A. Horawa]{Aleksander Horawa}
\address{Mathematical Institute\\University of Bonn\\
Endenicher Allee 60 \\ 53115 Bonn, Germany}
\email{horawa@math.uni-bonn.de}
\author[S. D. Li-Huerta]{Siyan Daniel Li-Huerta}
\address{Department of Mathematics\\Massachusetts Institute of Technology\\77 Massachusetts Avenue\\Cambridge, MA 02139}
\email{sdlh@mit.edu}
\author[N. Sweeting]{Naomi Sweeting}
\address{Department of Mathematics\\Princeton University\\Fine Hall, Washington Road\\Princeton, NJ 08540}
\email{naomiss@math.princeton.edu}

\date{}

\begin{document}

\maketitle

\begin{abstract}
    Quaternionic modular forms on $\mathsf{G}_2$ carry a surprisingly rich arithmetic structure. For example, they have a theory of Fourier expansions where the Fourier coefficients are indexed by totally real cubic rings. For quaternionic modular forms on $\mathsf{G}_2$ associated via functoriality with certain modular forms on $\mathrm{PGL}_2$, Gross conjectured in 2000 that their Fourier coefficients encode $L$-values of cubic twists of the modular form (echoing Waldspurger's work on Fourier coefficients of half-integral weight modular forms). We prove Gross's conjecture when the modular forms are dihedral, giving the first examples for which it is known.
\end{abstract}

\setcounter{tocdepth}{1}
\tableofcontents

\section{Introduction}
To any holomorphic modular cusp form $f$ of even weight $2k$, one can associate its Shimura lift $\mathcal{F}$ \cite{Shi73}, which is a holomorphic modular cusp form of weight $k+\frac12$. Waldspurger \cite{Waldspurger_Sur_les_coeff} discovered a remarkable relationship between the Fourier coefficients of $\mathcal{F}$ and the $L$-values of quadratic twists of $f$. This led, for example, to Tunnell's partial resolution \cite{Tunnell} of the congruent number problem. 

The goal of this paper is to prove a similar relationship between the Fourier coefficients of certain \emph{quaternionic modular forms} on $\mathsf{G}_2$ and the $L$-values of \emph{cubic} twists of certain holomorphic modular forms. This was conjectured by Gross in 2000, and we prove his conjecture for the first class of examples: the dihedral case.

\subsection{Quaternionic modular forms on $\mathsf{G}_2$}
Let $G$ be a connected reductive group over $\Q$. When the symmetric space associated with $G(\RR)$ is Hermitian, there is a natural generalization of holomorphic modular forms to $G$: automorphic forms on $G$ that generate a holomorphic discrete series over $G(\RR)$. These automorphic forms are well-known to have rich connections with arithmetic.

When $G$ is the split simple group of type $\mathsf{G}_2$, the real Lie group $G(\RR)$ does \emph{not} have holomorphic discrete series. Nonetheless, Gross--Wallach \cite{GW96} singled out a class of representations $\{\pi_k\}_{k\geq1}$ of $G(\RR)$ called \emph{quaternionic discrete series}\footnote{Strictly speaking, $\pi_1$ is only a limit of quaternionic discrete series, but this does not matter for our purposes.}, and Gan--Gross--Savin \cite{GGS02} initiated the arithmetic study of \emph{quaternionic modular forms (of weight $k$)}, that is, automorphic forms $\mathcal{F}$ on $G$ that generate $\pi_k$ over $G(\RR)$.

Pollack \cite{MR4094735} developed the following explicit theory of Fourier expansions for quaternionic modular forms. The group $G$ has a Heisenberg parabolic with Levi subgroup $M\cong\GL_2$ and unipotent radical $N$. Write $Z$ for the center of $N$, write $\X$ for $\Hom(N/Z,\G_a)$, and write $\langle-,-\rangle:\X\times(N/Z)\ra\G_a$ for the evaluation pairing. For all $\mathcal{E}$ in $\X(\Q)$, Pollack defines an explicit function $\mathcal{W}^{\mathcal{E}}:M(\RR)\ra\C$ such that, for any quaternionic modular form~$\mathcal{F}$, its $Z$-constant term $\mathcal{F}_Z(g)\coloneqq\int_{Z(\Q)\backslash Z(\A)}\mathcal{F}(zg)\dd{g}$ can be written as
\begin{align}\label{eqn:intro_FE}
    \mathcal F_Z (n g) &  = \sum_{\mathcal E \in \mathbb X(\Q)} a_{\mathcal E}(\mathcal F) e^{-2\pi i \langle \mathcal E, n \rangle} \mathcal{W}^{\mathcal E}(g) \quad\mbox{for all } g \in M(\RR)\mbox{ and }n \in N(\RR),
\end{align}
where the $a_{\mathcal{E}}(\mathcal{F})$ lie in $\C$ \cite[Corollary 1.2.3]{MR4094735}. Since the representation $\X$ of $M\cong\GL_2$ is isomorphic to $\Sym^3\otimes\det^{-1}$, a classic result of Delone--Faddeev \cite{DF40} shows that $M(\Q)$-orbits in $\X(\Q)$ correspond to cubic algebras over $\Q$. In fact, their work refines to show that $M(\Z)$-orbits in $\X(\Z)$ correspond to cubic algebras over $\Z$. When $\mathcal{E}\in\X(\Q)$ corresponds to an \'etale cubic algebra $E/\Q$, one can show that $a_\mathcal{E}(\mathcal{F})$ vanishes unless $E$ is totally real.

Examples of quaternionic modular forms include certain Eisenstein series, whose Fourier coefficients $a_{\mathcal{E}}(\mathcal{F})$ have been studied extensively by Jiang--Rallis \cite{Jiang_Rallis}, Gan--Gross--Savin \cite{GGS02}, and Xiong \cite{Xiong}.

\subsection{Gross's conjecture}
What about \emph{cuspidal} examples of quaternionic modular forms? For any holomorphic modular cusp form $f$ of even weight $2k$ with level $1$ and trivial character, Arthur's conjecture \cite{Arthur} predicts a cuspidal quaternionic modular form $\mathcal{F}$ of weight $k$ on $G$ associated with $f$ by Langlands functoriality. In a manner analogous to Shimura lifts, Gan--Gurevich \cite{Gan_Gurevich} gave a conjectural construction of~$\mathcal{F}$, assuming that $L(\frac12,f)\neq0$.

Gross conjectured the following analogue of Waldspurger's theorem \cite{Waldspurger_Sur_les_coeff}:
\begin{conj}[Gross \cite{ChaoLi}]\label{conj:Gross}
Assume that $f$ has level $1$. For all $\mathcal{E}\in\X(\Z)$ corresponding to the ring of integers of a totally real \'etale cubic algebra $E/\Q$, we have
\begin{align*}
    a_{\mathcal{E}}(\mathcal{F})^2 = L(\textstyle \frac 1 2, f \otimes V_E) \cdot \Delta_{E}^{k-\frac 1 2},
\end{align*}
where $V_E$ is the $2$-dimensional Artin representation with $\Ind_E^\Q\one=\one\oplus V_E$, and $\Delta_E$ is the discriminant of $E$.
\end{conj}
Since $f$ has level $1$, the form $\mathcal F$ is invariant under $G(\widehat{\mathbb{Z}})$, which implies that $a_{\mathcal{E}}(\mathcal{F})$ vanishes unless $\mathcal{E}\in\X(\Z)$.

More generally, for any holomorphic modular cusp form $f$ with trivial character, Arthur's conjecture \cite{Arthur} predicts \emph{multiple} cuspidal quaternionic modular forms on $G$ associated with $f$.

Our paper studies the case where $f$ is dihedral. Namely, let $K/\Q$ be an imaginary quadratic extension, let~$\chi$ be a conjugate-symplectic Hecke character for $K$ with $L(\frac12,\chi)\neq0$, and take $f$ to be the associated dihedral modular cusp form. Then the level $N$ of $f$ must be nontrivial; in fact, $N$ is necessarily not squarefree.

Here, Arthur's conjecture \cite{Arthur} predicts that, for every sequence $\epsilon=(\epsilon_p)_p$ in $\{\pm1\}$ indexed by primes $p$ with
\begin{itemize}
    \item $\epsilon_p=+1$ when $p$ splits in $K$ or $\chi_p^2=1$ (which includes all $p$ not dividing $N$),
    \item $\prod_p\epsilon_p=-\epsilon(\frac12,\chi^3)$,
\end{itemize}
there should be a cuspidal automorphic representation $\pi^\epsilon$ of $G$ associated with $f$ whose $p$-adic component is explicitly determined by $\epsilon_p$ and $\chi_p$ and whose archimedean component is $\pi_k$. In our previous work \cite{ourselves}, we proved Arthur's conjecture in this case (assuming that $K/\Q$ is unramified at $2$); in particular, we gave an unconditional definition of $\pi^\epsilon$. 

For any quaternionic modular form $\mathcal{F}^\epsilon\in\pi^\epsilon$ and for all $\mathcal{E}\in\X(\Q)$ corresponding to an \'etale cubic algebra $E/\Q$, we show that $a_{\mathcal{E}}(\mathcal{F}^\epsilon)$ vanishes for local reasons unless
\begin{itemize}
    \item $\epsilon_p=\epsilon_p(E_p,\chi_p)$ for all primes $p$, where $\epsilon_p(E_p,\chi_p)\in\{\pm1\}$ is purely local (see Definition \ref{defn:cubicepsilon}\footnote{While Definition \ref{defn:cubicepsilon} depends on a continuous character $\psi_p:\Q_p\ra\C^\times$, in the introduction we fix $\psi:\Q\backslash\A\ra\C^\times$ to be the unique continuous character such that $\psi_\infty(x)=e^{-2\pi i x}$.}),
    \item $E$ is totally real (which is the archimedean analogue of the above $\epsilon$-condition).
\end{itemize}
When $p$ does not divide $N$, we show that $\epsilon_p(E_p,\chi_p)=+1$ for all \'etale cubic algebras $E_p/\Q_p$; this explains why the above $\epsilon$-condition does not appear in Conjecture \ref{conj:Gross}.

We construct a quaternionic modular form $\mathcal{F}^{\epsilon}\in\pi^\epsilon$ satisfying the following version of Conjecture \ref{conj:Gross}, which takes into account the aforementioned local obstructions:
\begin{thm}[Theorem \ref{thm:refined}]\label{thm:main}Assume that $L(\frac12,f)\neq0$. For all $\mathcal{E}\in\X(\Q)$, the Fourier coefficient $a_{\mathcal{E}}(\mathcal{F}^\epsilon)$ vanishes unless $\mathcal{E}\in\X(\Z)$. Moreover, if $\mathcal{E}\in\X(\Z)$ corresponds to the ring of integers of a totally real \'etale cubic algebra $E/\Q$ such that $\epsilon_p=\epsilon_p(E_p,\chi_p)$ for all $p$ dividing $N$, then
\begin{align*}
    |a_{\mathcal{E}}(\mathcal{F}^{\epsilon})|^2 = L(\textstyle \frac 1 2, f \otimes V_E) \cdot \Delta_{E}^{k-\frac 1 2}.
\end{align*}
\end{thm}
\begin{remark}
Conjecture \ref{conj:Gross} does not take absolute values, while we \emph{do} take absolute values in Theorem \ref{thm:main}. This is essential for our method, as we explain in \S\ref{ss:proofsketch} below. Something similar happens when extracting explicit results from Waldspurger's theorem \cite{Waldspurger_Sur_les_coeff}: precise formulas which calculate all the constants of proportionality all take absolute values \cite{Kohnen_Zagier, Kohnen, Baruch_Mao, Prasanna_On_the_Fourier}.
\end{remark}
\begin{remark}
We actually work over totally real fields $F$ and prove a generalization of Theorem \ref{thm:main} to dihedral Hilbert modular forms (see Theorem \ref{thm:refined}). We restrict to $F=\Q$ here for the sake of exposition.
\end{remark}

\subsection{Related work}
Fourier coefficients of cuspidal quaternionic modular forms have been extensively studied by Pollack; for example, he proved that, for all $k\geq6$, there exists a basis of cuspidal quaternionic modular forms on $G$ whose Fourier coefficients lie in $\Q^{\operatorname{ab}}$ \cite[Theorem 1.0.1]{Pollack_G_2}.

Let us explain why this is consistent with Theorem \ref{thm:main}. The projection formula yields
\begin{align*}
\Ind_K^\Q \chi \otimes \Ind_E^\Q \one = \Ind_{KE}^{\Q}(\chi|_{KE}),
\end{align*}
so we get $(\Ind_K^\Q \chi) \otimes V_E = \Ind_{KE}^\Q(\chi|_{KE})-\Ind_{K}^\Q \chi$. Taking $L$-functions gives
\begin{align*}
L({\textstyle\frac12}, f \otimes V_E) = \frac{L(\textstyle\frac{1}{2}, \chi|_{KE})}{L(\frac{1}{2}, \chi)}.
\end{align*}
Blasius's work \cite{Blasius} on Deligne's conjecture shows that there are periods $c^+(\chi), c^+(\chi|_{KE}) \in \C$ such that
\begin{align*}
\frac{L(\textstyle\frac{1}{2}, \chi)}{c^+(\chi)}, \frac{L(\frac{1}{2}, \chi|_{KE})}{c^+(\chi|_{KE})} \in \Q(\chi).
\end{align*}
If the motive for $\chi$ is realized in an abelian variety $A/\Q$ with CM by $K$, then the motive for $\chi|_{KE}$ is realized in $A/E$, so one can show that $c^+(\chi|_{KE})=c^+(\chi)^3$. Therefore we get
\begin{align*}
\frac{L(\textstyle\frac12,f\otimes V_E)}{c^+(\chi)^2}\in\Q(\chi)\subseteq\Q^{\operatorname{ab}}.
\end{align*}
Because $c^+(\chi)^2$ is independent of $E$, this yields the desired consistency with Pollack's result.

Pollack \cite[Corollary 1.2.4]{Pollack_G_2} also obtained the first result towards Gross's conjecture: when $f$ is the cusp form $\Delta$ of weight 12, Conjecture \ref{conj:Gross} is true when $\mathcal E$ corresponds to the ring of integers of $\Q \times F'$ for a totally real \'etale quadratic algebra $F'$ over $\Q$.

Recently, assuming Arthur's conjecture \cite{Arthur}, Kim--Yamauchi \cite[Theorem 1.4]{Kim_Yamauchi} generalized Pollack's \cite[Corollary 1.2.4]{Pollack_G_2} to all $f$ with squarefree level. Because the $f$ that we consider do not have squarefree level, their results are disjoint from ours. Their methods are also quite different: they only consider the case where $\epsilon_p=+1$ for all $p$, and by using explicit models for $\pi^+_p$ and studying the Fourier--Jacobi expansion of $\pi^\epsilon$ along the \emph{other} maximal parabolic subgroup of $G$, they relate $a_{\mathcal{E}}(\mathcal{F}^\epsilon)$ to the $D$-th Fourier coefficient of the Shimura lift of $f$ when $\mathcal{E}$ corresponds to the ring of integers of $\Q\times\Q(\sqrt{D})$. Finally, they relate the latter to $L(\frac12,f\otimes V_{\Q\times\Q(\sqrt{D})})=L(\frac12,f)L(\frac12,f\otimes\chi_D)$ using Waldspurger's theorem \cite{Waldspurger_Sur_les_coeff}.

\subsection{Idea of proof}\label{ss:proofsketch}
Let us explain the proof of Theorem \ref{thm:main}. We start with the construction of $\mathcal{F}^\epsilon$: it is an \emph{exceptional theta lift} from the compact form $G'$ of $\PU_3$ with respect to $K/\Q$, using theta kernels on the quasi-split adjoint form 
$\widetilde{G}$ of $\mathsf{E}_6$ with respect to $K/\Q$. More precisely, we associate a cuspidal automorphic representation $\sigma^\epsilon$ of $G'$ to $\epsilon$ and $\chi$, and for all $f'$ in $\sigma^\epsilon$ and $\varphi$ in the minimal representation $\Omega$ of $\widetilde{G}$, we construct a quaternionic modular form $\mathcal{F}^\epsilon\coloneqq\theta(\varphi,f')\in\pi^\epsilon$.

The Fourier coefficient $a_{\mathcal{E}}(\mathcal{F}^\epsilon)$ is explicitly related to the automorphic Fourier coefficient $\theta(\varphi,f')_{N,\psi_{\mathcal{E}}}(1)$, where $\psi_{\mathcal{E}}$ denotes the continuous character of $N(\A)$ associated with $\mathcal{E}$. Write $\widetilde{N}$ for the unipotent radical of the Heisenberg parabolic of $\widetilde{G}$. By studying the automorphic Fourier coefficients of the theta kernels $\theta(\varphi)$ along $\widetilde{N}$, we prove that
\begin{align}\label{eq:firstmaineq}
    \theta(\varphi,f')_{N,\psi_{\mathcal{E}}}(1)=\int_{i(T_E)(\A)\backslash G'(\A)}\theta(g'\cdot\varphi)_{\widetilde{N},\psi_i}(1)\overline{\mathcal{P}_i(g'\cdot f')}\dd{g'},
\end{align}
where $i:T_E\hookrightarrow G'$ is a certain maximal subtorus associated with the \'etale cubic algebra $E/\Q$, $\psi_i$ is a certain continuous character of $\widetilde N(\A)$ restricting to $\psi_{\mathcal{E}}$ on $N(\A)$, and
\begin{align*}
\mathcal{P}_i(f')\coloneqq\int_{i(T_E)(\Q)\backslash i(T_E)(\A)}f'(t')\dd{t'}
\end{align*}
is the period on $\sigma^\epsilon$ associated with $i:T_E\hookrightarrow G'$.

For certain special factorizable $f_i'$ in $\sigma^\epsilon$, we relate $|\mathcal{P}_i(f_i')|^2$ to $L(\frac12,f\otimes V_E)\cdot\Delta_E^{1/2}$ by combining a seesaw of (classical) unitary group theta lifts with explicit calculations of T. Yang \cite{Yan97}. To leverage this relationship, we prove that \eqref{eq:firstmaineq} also has a factorizable form
\begin{align*}
\theta(\varphi,f')_{N,\psi_{\mathcal{E}}}(1) = \mathcal{P}_i(f_i')\cdot\prod_v\mathcal{I}_v(\mathcal{E}_v,\varphi_v,f_v')
\end{align*}
whenever $\varphi=\otimes'_v\varphi_v$ and $f'=\otimes'_vf_v'$ are factorizable, where the $\mathcal{I}_v(\mathcal{E}_v,\varphi_v,f_v')$ are certain local integrals that incorporate both the discrepancy between $f'_v$ and $f'_{i,v}$ as well as the local Fourier coefficients of the local minimal representation $\Omega_v$. Hence it remains to compute $\mathcal{I}_v(\mathcal{E}_v,\varphi_v,f'_v)$ for appropriate choices of $\varphi_v$ and $f_v'$.

At $p$-adic places where $p$ does not divide $N$, we take $\varphi_p$ and $f'_p$ to be normalized spherical vectors in $\Omega_p$ and in the $p$-adic component $\sigma^{\epsilon_p}_p$ of $\sigma^\epsilon$, respectively. We prove that $f'_p$ is an (unspecified) $\C^1$-multiple of an $i(T_E)(\Q_p)$-translate of $f'_{i,p}$, which lets us reduce the computation of $|\mathcal{I}_p(\mathcal{E}_p,\varphi_p,f'_p)|$ to the following elementary statement (Lemma \ref{lem:double_coset}) and its Hermitian analogue:
\begin{center}
for all $p$-adic fields $F_p$, \'etale cubic algebras $E_p/F_p$, and $\O_{F_p}$-algebra injections $i:\O_{E_p}\hookrightarrow\operatorname{M}_3(\O_{F_p})$,\\
if $g$ in $\GL_3(F_p)$ satisfies $g^{-1}i(\O_{E_p})g\subseteq\operatorname{M}_3(\O_{F_p})$, then $g$ lies in $i(E^\times_p)\GL_3(\O_{F_p})$.
\end{center}

At the archimedean place, the local minimal representation $\Omega_\infty$ is actually a limit of quaternionic discrete series for the ambient group $\widetilde{G}(\RR)$, so we can study it using the work of Pollack \cite{MR4094735}. To ensure that $\theta_\infty(\varphi_\infty,f_\infty')$ is a highest weight vector in the lowest $K$-type of $\pi_k$, we first define a raising operator $\mathcal{D}^+_k$ in $\widetilde{\g}_\C$ and then take $\varphi_\infty$ to be $\mathcal{D}^+_k$ applied to a normalized highest weight vector in the lowest $\widetilde{K}$-type of $\Omega_\infty$. Since $G'(\RR)$ is compact, we can take $f'_\infty$ to be a normalized highest weight vector in the archimedean component $\sigma_\infty$ of~$\sigma^\epsilon$. With these choices, we compute $\mathcal{I}_\infty(\mathcal{E}_\infty,\varphi_\infty,f'_\infty)$ using work of Pollack \cite{MR4094735}. The result (Theorem \ref{thm:arch_integral}) crucially involves $\Delta_E^{(k-1)/2}$ as well as some factors that cancel with the definition of $\mathcal{W}^{\mathcal{E}}$ from \eqref{eqn:intro_FE}.

The above work already suffices to prove (Theorem \ref{thm:general}\footnote{There are also some local constants in Theorem \ref{thm:general}, which we ignore here for simplicity; alternatively, one can renormalize the definition of $\mathcal{I}_p(\mathcal{E}_p,\varphi_p,f'_p)$ to incorporate these constants.}) that, under the necessary local conditions,
\begin{align*}
|a_{\mathcal{E}}(\mathcal{F}^\epsilon)|^2 = L(\textstyle\frac12,f\otimes V_E)\cdot\Delta_E^{k-\frac  12}\cdot\displaystyle\prod_{p\mid N}|\mathcal{I}_p(\mathcal{E}_p,\varphi_p,f'_p)|^2.
\end{align*}
Finally, at $p$-adic places where $p$ divides $N$, we custom design $\varphi_p$ and $f'_p$ so that $\mathcal{I}_p(\mathcal{E}_p,\varphi_p,f'_p)$ equals $1$ when $\epsilon_p=\epsilon_p(\chi_p,E_p)$ and equals $0$ otherwise (Proposition \ref{prop:ramified_local_integrals}). This concludes the proof of Theorem \ref{thm:main}.

\subsection*{Outline}
In \S\ref{s:Yang}, we introduce the automorphic forms on $\PU_3$ that we will lift to $\mathsf{G}_2$ and compute their relevant torus periods. In \S\ref{s:formulas}, we define Fourier coefficients for quaternionic modular forms on $\mathsf{G}_2$, gather facts about the exceptional theta lift between $\PU_3$ and $\mathsf{G}_2$, and prove our main results modulo calculating certain local integrals. We calculate these local integrals at $p$-adic places where $p$ does not divide $N$ in \S\ref{s:unramified}, at $p$-adic places where $p$ divides $N$ in \S\ref{s:ramified}, and at archimedean places in \S\ref{s:archimedean}.

\subsection*{Notation}
Throughout this paper, $F$ is a field of characteristic $0$, and $K$ is a quadratic \'etale $F$-algebra. Write $k\mapsto\overline{k}$ for the nontrivial element of $\Aut_F(K)\cong\mathbb{Z}/2$.

When $F$ is a nonarchimedean local field, write $v$ for its normalized valuation, and write $\varpi$ for a choice of uniformizer. When $F$ is an archimedean local field, we always assume that $F=\RR$ and $K=\C$. We use the absolute value on $\C$ given by $z\mapsto\sqrt{z\overline{z}}$. Whenever possible and unless otherwise specified, all Haar measures give maximal compact subgroups volume $1$.

When $F$ is a number field, we always assume that $F$ is totally real and $K$ is totally imaginary. For any affine algebraic group $G$ over $F$, write $[G]$ for $G(F)\backslash G(\A_F)$. Our automorphic representations are all irreducible, contrary to our convention in \cite{ourselves}.

\subsection*{Acknowledgments}
We are indebted to Wee Teck Gan, Dick Gross, and Gordan Savin for initiating the arithmetic study of quaternionic modular forms on $\mathsf{G}_2$. In addition, we are extremely thankful to Wee Teck Gan and Dick Gross for suggesting this problem and for their continued interest and encouragement. Finally, we are very grateful to Aaron Pollack for helpful discussions.

The second-named author was supported by UK Research and Innovation grant MR/V021931/1. The third-named author was partially supported by NSF Grant \#DMS2303195 and the Max Planck Institute for Mathematics. The fourth-named author was supported by NSF grant \#DMS2401823.

For the purpose of Open Access, the authors have applied a CC BY public copyright licence to any Author Accepted Manuscript (AAM) version arising from this submission.

\section{Unitary group theta lifts}\label{s:Yang}
We will construct and study automorphic forms on $\mathsf{G}_2$ using an exceptional theta lift from $\PU_3$, so in this section we gather the necessary results about $\PU_3$. In \S\ref{ss:unitarysetup}, we begin with basic notation on $3$-dimensional Hermitian spaces, as well as their relation with cubic algebras. In \S\ref{ss:seesaw}, we introduce our unitary group theta lifts and seesaw, the latter of which is essential for our results on torus periods.

We take a break in \S\ref{ss:localPU3} to define spherical vectors in our local representations for later use. In \S\ref{ss:G'localvector} and \S\ref{ss:actuallocalvector}, we return to our torus periods and study them in the local context. Finally, we study the analogous global torus period in \S\ref{ss:Yang}. Our work relies on explicit calculations of T. Yang \cite{Yan97}.

This section can be viewed as a refinement of \cite[\S3]{ourselves} in the setting of this paper.

\subsection{Unitary groups and algebra embeddings}\label{ss:unitarysetup}
We begin by setting up our $3$-dimensional Hermitian spaces and explaining their relationship with Freudenthal--Jordan algebras.

Equip $K^3$ with the Hermitian form for $K/F$ given by $(v_1,v_2)\mapsto v_1\cdot\overline{v_2}$. Write $\U_3$ for its associated unitary group over $F$, and write $G'$ for the adjoint group of $\U_3$. Note that the discriminant of $K^3$ equals the image of $-1$ under $F^\times\ra F^\times/\Nm_{K/F}(K^\times)$.

For all $x$ and $y$ in $\operatorname{M}_3(K)$, write $x\circ y$ for $\frac12(xy+yx)$, and write $x^\#$ for the adjugate matrix of $x$. Write $J$ for the set of Hermitian matrices in $\operatorname{M}_3(K)$, and use $\circ$ to equip $J$ with the structure of a Freudenthal algebra over $F$ in the sense of \cite[\S37.C]{KMRT}. Now $G'$ acts on $J$ via conjugation, which identifies $G'$ with the connected automorphism group of $J$ over $F$.

Let $E$ be a cubic \'etale $F$-algebra, and write $L$ for $E\otimes_FK$. Write $T_E$ for the $2$-dimensional torus
\begin{align*}
\coker(\R^1_{K/F}\mathbb{G}_m\ra\R_{E/F}(\R^1_{L/E}\mathbb{G}_m))
\end{align*}
over $F$. Because $F^\times/\Nm_{K/F}(K^\times)$ is $2$-torsion and the composition
\begin{align*}
\xymatrix{F^\times/\Nm_{K/F}(K^\times)\ar[r] & E^\times/\Nm_{L/E}(L^\times)\ar[r]^-{\Nm_{E/F}} & F^\times/\Nm_{K/F}(K^\times)}
\end{align*}
is the cubing map, the left arrow is injective. Therefore $(\R^1_{L/E}\G_m)(E)\ra T_E(F)$ is surjective.

Write $\{E\hookrightarrow J\}$ for the set of $F$-algebra embeddings $E\hookrightarrow J$. Note that any $i$ in $\{E\hookrightarrow J\}$ induces an injective morphism $i:\R_{E/F}(\R^1_{L/E}\mathbb{G}_m)\hookrightarrow\U_3$ and hence $i:T_E\hookrightarrow G'$ of groups over $F$. Moreover, we see that the stabilizer of $i$ in $G'$ equals $i(T_E)$.

Let $\lambda$ be in $E^\times$. Write $L_\lambda$ for the associated $1$-dimensional Hermitian space for $L/E$, so that $\R^1_{L/E}\G_m$ is its associated unitary group over $E$. By postcomposing the Hermitian form with $\tr_{L/K}$, we view $L_\lambda$ as a $3$-dimensional Hermitian space for $K/F$. If we have an isomorphism $K^3\cong L_\lambda$ of Hermitian spaces for $K/F$, then we obtain an embedding $L\hookrightarrow\operatorname{M}_3(K)$ of $K$-algebras with involution and hence an element of $\{E\hookrightarrow J\}$. By \cite[Lemma 3.2]{ourselves}, every $i$ in $\{E\hookrightarrow J\}$ arises from this construction for some $\lambda$ in $E^\times$ and some isomorphism $K^3\cong L_\lambda$ of Hermitian spaces for $K/F$, and $\lambda_1$ and $\lambda_2$ in $E^\times$ induce the same $G'(F)$-orbit in $\{E\hookrightarrow J\}$ if and only if $\lambda_1\lambda_2^{-1}$ lies in $F^\times\Nm_{L/E}(L^\times)$.

\subsection{The seesaw and local representations of $G'$}\label{ss:seesaw}
We now introduce our unitary group theta lift over~$F$, which we will use to construct representations of $G'$. This theta lift fits into a seesaw with a unitary group theta lift over $E$, which we will use to compute torus periods of our representations of $G'$.

Let $\epsilon_1$ be in $F^\times$. Write $K_{\epsilon_1}$ for the associated $1$-dimensional Hermitian space for $K/F$, and write $\U_1$ for its associated unitary group over $F$. Let $\delta$ in $K^\times$ satisfy $\tr_{K/F}\delta=0$, and write $W$ for the symplectic space $K_{\epsilon_1}\otimes_KK^3=K^3$ over $F$, where the symplectic form is given by
\begin{align*}
(w_1,w_2)\mapsto\tr_{K/F}(\delta\epsilon_1w_1\cdot\overline{w_2}).
\end{align*}
Note that $W$ has a polarization given by $F^3\oplus\delta F^3$. When $K=F\times F$, it also has a polarization given by $(1,0)F^3\oplus(0,1)F^3$.

Given a $\lambda$ in $E^\times$ and an isomorphism $K^3\cong L_\lambda$ of Hermitian spaces for $K/F$ that give rise to $i$ as in \S\ref{ss:unitarysetup}, we get an isomorphism between $W$ and the analogous symplectic space over $F$ induced by $\delta$ and $L_{\epsilon_1}\otimes_LL_\lambda=L$. Then the analogue of \cite[(2.17)]{kudla1984seesaw} for unitary groups yields a seesaw of dual pairs in $\operatorname{Sp}_W$ over $F$
\begin{align}\label{eq:seesaw}
\begin{split}\xymatrix{\R_{E/F}\R^1_{L/E}\G_m\ar@{-}[rd] & \U_3\ar@{-}[ld]\\
\U_1\ar[u] & \R_{E/F}\R^1_{L/E}\G_m\ar[u]_-i.}\end{split}
\end{align}
Note that $L_{\epsilon_1}\otimes_LL_{\lambda}$ has a polarization given by $E\oplus\delta E$.

For the rest of this subsection, assume that $F$ is a local field or a number field. Write
\begin{align*}
C_F\coloneqq\begin{cases}
F^\times & F\mbox{ local field,}\\
F^\times\backslash\A_F^\times & F\mbox{ number field,}
\end{cases} & & A_F\coloneqq\begin{cases}
F & F\mbox{ local field,}\\
F\backslash\A_F & F\mbox{ number field,}
\end{cases}
& & R_F\coloneqq\begin{cases}
F & F\mbox{ local field,} \\
\A_F & F\mbox{ number field.}
\end{cases}
\end{align*}
Let $\chi:C_K\ra \C^1$ be a conjugate-symplectic unitary character, and let $\psi:A_F\ra \C^1$ be a nontrivial unitary character.
By \cite[Theorem 3.1]{Kud87}, the data of
\begin{itemize}
    \item $\psi$ and $(\chi,\chi^3)$ induces a lifting of $(\U_1\times\U_3)(R_F)\rightarrow\operatorname{Sp}_W(R_F)$ to $\operatorname{Mp}_W(R_F)$,
    \item $\psi\circ\tr_{E/F}$ and $(\chi\circ\Nm_{L/K},\chi\circ\Nm_{L/K})$ induces a lifting of $(\R^1_{L/E}\G_m\times\R^1_{L/E}\G_m)(R_E)\rightarrow\operatorname{Sp}_W(R_F)$ to $\operatorname{Mp}_W(R_F)$.
\end{itemize}
Since $\Nm_{L/K}(k)=k^3$ for all $k$ in $K^\times$, these lifts restrict to one another under (\ref{eq:seesaw}).

Endow $R_F^3$ with the self-dual measure with respect to $\psi$, which yields a Hermitian pairing on the Schwartz space $\mathcal{S}(R_F^3)$. Using the polarization $W=F^3\oplus \delta F^3$, equip $\mathcal{S}(R_F^3)$ with the Weil representation of $\operatorname{Mp}_W(R_F)$ associated with $\psi$, and write $\theta_{13}(-)$ for the resulting theta lift from $\U_1$ to $\U_3$.

In the setting of (\ref{eq:seesaw}), endow $R_E$ with the self-dual measure with respect to $\psi\circ\tr_{E/F}$, which yields a Hermitian pairing on $\mathcal{S}(R_E)$. Using the polarization $L_{\epsilon_1}\otimes_LL_\lambda=E\oplus\delta E$, equip $\mathcal{S}(R_E)$ with the Weil representation of $\operatorname{Mp}_W(R_F)$ associated with $\psi$, and write $\theta_E(-)$ for the resulting theta lift from $\R_{E/F}\R^1_{L/E}\G_m$ to $\R_{E/F}\R^1_{L/E}\G_m$ (when $F$ is a number field, assume that $\R^1_{L/E}\G_m$ is anisotropic over~$E$). When $F$ is a number field, any element of $\operatorname{Sp}_W(F)$ that sends the polarization $F^3\oplus\delta F^3$ to $E\oplus\delta E$ induces a unitary isomorphism $\mathcal{S}(\A_F^3)\cong\mathcal{S}(\A_E)$ of representations of $\operatorname{Mp}_W(\A_F)$ that preserves their natural automorphic realizations.

For the rest of this subsection, assume that $F$ is a local field. Write $\omega_{K/F}:F^\times\ra\{\pm1\}$ associated with $K/F$ by class field theory, and consider the sign
\begin{align*}
\epsilon & := \omega_{K/F}(-\epsilon_1)\cdot\epsilon(\textstyle\frac12,\chi^3,\psi(\tr_{K/F}(\delta-)))\in\{\pm1\}.
\end{align*}
When $\chi^2=1$, assume that $\epsilon=+1$, and when $F$ is archimedean, assume that $\epsilon=-1$. 
\begin{definition}\label{defn:localsigma}
Write $\sigma^{\epsilon}$ for the irreducible smooth representation $\theta_{13}(\one)$ of $G'(F)$.
\end{definition}
By \cite[Proposition 3.7]{ourselves} or \cite[Proposition 3.9]{ourselves}, Definition \ref{defn:localsigma} agrees with the $\sigma^\epsilon$ from \cite[\S3.4]{ourselves} or \cite[\S3.5]{ourselves}.

\subsection{Spherical vectors for $G'$}\label{ss:localPU3}
In this subsection, we define spherical vectors in our representations of~$G'$ for later use. Assume that $F$ is a nonarchimedean local field, $K/F$ and $\chi$ are unramified, and $\psi$ has conductor $0$.

When $K$ is a field, our unramified hypotheses imply that $\chi^2=1$, so $\epsilon=+1$ by assumption. Therefore $v_K(\delta\epsilon_1)$ must be even; write $n\coloneqq v_K(\delta\epsilon_1)/2$. Then $\varpi^{-n}\O_K^3$ is self-dual in $W$ and compatible with the polarization $W=F^3\oplus\delta F^3$. Write $\phi_0$ in $\mathcal{S}(F^3)$ for $\vol(\varpi^{-n}\O_F^3)^{-1/2}$ times the indicator function of $\varpi^{-n}\O_F^3$.

When $K=F\times F$, endow $(1,0)F^3=F^3$ with the self-dual measure with respect to $\psi$, which yields a Hermitian pairing on $\mathcal{S}((1,0)F^3)$. Using the polarization $W=(1,0)F^3\oplus(0,1)F^3$, equip $\mathcal{S}((1,0)F^3)$ with the Weil representation of $\operatorname{Mp}_W(F)$ associated with $\psi$. Fix a unitary isomorphism $\mathcal{S}((1,0)F^3)\cong\mathcal{S}(F^3)$ of representations of $\operatorname{Mp}_W(F)$, and write $\phi_0$ in $\mathcal{S}(F^3)$ for the image of the indicator function of $(1,0)\O_F^3$.

Finally, write $f_0$ for the image of $\phi_0$ under the theta lift map $\mathcal{S}(F^3)\ra\one\boxtimes\theta_{13}(\one)=\one\boxtimes\sigma^\epsilon\cong\sigma^\epsilon$. Observe that $f_0$ is indeed $G'(\O_F)$-fixed.

\subsection{Local torus periods}\label{ss:G'localvector}
In this subsection, assume that $F$ is a local field. Our local torus periods are controlled by the following invariants. Identify $E^\times/\Nm_{L/E}(L^\times)$ with its image in $\{\pm1\}^{\pi_0(\operatorname{Spec}{E})}$ under $\omega_{L/E}$, and consider the following element of $E^\times/\Nm_{L/E}(L^\times)$:
\begin{align*}
\lambda_0\coloneqq\Delta_{E/F}\cdot\epsilon(\textstyle\frac12,\chi\circ\Nm_{L/K},\psi(\tr_{L/F}(\delta-)))\cdot\Nm_{E/F}\!\Big[\epsilon(\textstyle\frac12,\chi\circ\Nm_{L/K},\psi(\tr_{L/F}(\delta-)))\Big],
\end{align*}
where $\Delta_{E/F}$ in $F^\times/(F^\times)^2$ denotes the discriminant of $E/F$, and the $\epsilon$-factors at $s=\frac12$ are interpreted as elements of $\{\pm1\}^{\pi_0(\operatorname{Spec}{E})}$.
\begin{definition}\label{defn:cubicepsilon}
Write $\epsilon(E,\chi,\psi)$ for the sign $\omega_{K/F}(-\Delta_{E/F}\Nm_{E/F}(\lambda_0))\cdot\epsilon(\textstyle\frac12,\chi^3,\psi(\tr_{K/F}(\delta-)))$ in $\{\pm1\}$.
\end{definition}
When $F$ is nonarchimedean and $K/F$ and $\chi$ are unramified, \cite[Proposition 6.8]{ourselves} shows that $\epsilon(E,\chi,\psi)=+1$.

When $F$ is archimedean, the proof of \cite[Proposition 6.15]{ourselves} shows that
\begin{align*}
\epsilon(E,\chi,\psi)=\begin{cases}
+1 & \mbox{when }E\cong\RR\times\C,\\
-1 & \mbox{when }E\cong\RR^3.
\end{cases}
\end{align*}

\begin{prop}\label{prop:epsilon_and_i}
   There exists at most one $G'(F)$-orbit of $i$ in $\{E\hookrightarrow J\}$ satisfying $\Hom_{i(T_E)(F)}(\sigma^\epsilon,\one)\neq0$. This $G'(F)$-orbit exists if and only if $\epsilon=\epsilon(E,\chi,\psi)$, and in this case it arises from $\lambda_0$ and an isomorphism $K^3\cong L_{\lambda_0}$ of Hermitian spaces for $K/F$, as in \S\ref{ss:unitarysetup}.
\end{prop}
\begin{proof}
    This follows from \cite[Proposition 3.8]{ourselves} and \cite[Proposition 3.10]{ourselves}; note that $\lambda_0$ differs from the $\lambda$ therein by an element of $F^\times$, chosen so that $K^3$ and $L_{\lambda_0}$ are already isomorphic as Hermitian spaces for $K/F$ (instead of after scaling $\lambda_0$ by an element of $F^\times$). 
\end{proof}
Assume a $G'(F)$-orbit as in Proposition \ref{prop:epsilon_and_i} exists, and fix an $i_0$ in it. Fix a representative of $\lambda_0$ in $E^\times$, and fix an isomorphism $K^3\cong L_{\lambda_0}$ of Hermitian spaces for $K/F$ that gives rise to $i_0$ as in \S\ref{ss:unitarysetup}, so that we get an isomorphism $W\cong L_{\epsilon_1}\otimes_LL_{\lambda_0}$ of symplectic spaces over $F$ as in \S\ref{ss:seesaw}.

\subsection{Local vectors for $G'$}\label{ss:actuallocalvector} In this subsection, assume that $F$ is a local field. We will normalize our local torus periods using the following vectors in our representations of $G'$.

Recall from \S\ref{ss:seesaw} that we endowed $\mathcal{S}(E)$ and $\mathcal{S}(F^3)$ with Hermitian pairings for which the Weil representation of $\operatorname{Mp}_W(F)$ is unitary. Fix a unitary isomorphism $\mathcal{S}(E)\cong\mathcal{S}(F^3)$ of representations of $\operatorname{Mp}_W(F)$, and write $\phi_{i_0}$ in $\mathcal{S}(F^3)$ for the image of the element of $\mathcal{S}(E)$ defined as $\phi_{\one,v}$ in \cite[p.~48]{Yan97} with respect to the quadratic\footnote{We warn the reader that \cite{Yan97} denotes the base field by $F_v$ and the quadratic \'etale algebra by $E_v$.} $E$-algebra $L$. Write $f_{i_0}$ for the image of $\phi_{i_0}$ under the theta lift map $\mathcal{S}(F^3)\ra\one\boxtimes\theta_{13}(\one)=\one\boxtimes\sigma^{\epsilon}\cong\sigma^\epsilon$.

Since the stabilizer of $i_0$ in $G'(F)$ equals $i_0(T_E)(F)$, we identify the $G'(F)$-orbit of $i_0$ with $G'(F)/i_0(T_E)(F)$. Write $\overline{(-)}:G'(F)\ra G'(F)/i_0(T_E)(F)$ for the quotient map, and fix a section $s$ of $\overline{(-)}$.
\begin{definition}\label{defn:localYang}
For all $i=x\cdot i_0$ in the $G'(F)$-orbit of $i_0$, write $\phi_i$ for $s(\overline{x})\cdot\phi_{i_0}$, and write $f_i$ for $s(\overline{x})\cdot f_{i_0}$.
\end{definition}

Recall from Proposition~\ref{prop:epsilon_and_i} that $\Hom_{i(T_E)(F)}(\sigma^\epsilon,\one)$ is $1$-dimensional.

\begin{lemma}\label{lemma:equiv_of_f_Xv}
For all $g'$ in $G'(F)$, $i$ in the $G'(F)$-orbit of $i_0$, and $\beta$ in $\Hom_{i(T_E)(F)}(\sigma^\epsilon,\one)$, we have
\begin{align*}
\beta(g'^{-1}\cdot f_{g'\cdot i}) = \beta(f_i).
\end{align*}
\end{lemma}
\begin{proof}
Let $x$ be an element of $G'(F)$ such that $i=x\cdot i_0$. Then $\beta\circ x$ is an element of $\Hom_{i_0(T_E)(F)}(\sigma^\epsilon,\one)$, and $x^{-1}g'^{-1}s(\overline{g'x})$ and $x^{-1}s(\overline{x})$ lie in $i_0(T_E)(F)$. Therefore
\begin{gather*}
\beta(g'^{-1}\cdot f_{g'\cdot i}) = \beta(x\cdot x^{-1}g'^{-1}s(\overline{g'x})\cdot f_{i_0}) = \beta(x\cdot f_{i_0}) = \beta(x\cdot x^{-1}s(\overline{x})\cdot f_{i_0}) = \beta(f_i).\qedhere
\end{gather*}
\end{proof}
\begin{remark}\label{rem:fieldcaseYang}
When $K$ is a field, $\theta_{13}(\one)$ is the summand of $\mathcal{S}_F$ where $\U_1(F)$ acts trivially. Moreover, if $\R^1_{L/E}\G_m$ is anisotropic over $E$, then $(\R^1_{L/E}\G_m)(E)$ acts trivially on $\phi_{i_0}$ by construction \cite[p.~48]{Yan97}. Because $(\R^1_{L/E}\G_m)(E)$ contains $\U_1(F)$, this shows that $\phi_{i_0}$ lies in $\theta_{13}(\one)$. In particular, $i_0(T_E)(F)$ acts trivially on $f_{i_0}=\phi_{i_0}$. Therefore in this situation the section $s$ is unnecessary for defining $f_i$, and we have $g'^{-1}\cdot f_{g'\cdot i}=f_i$ even before applying $\beta$. However, when $\R^1_{L/E}\G_m$ is not anisotropic over $E$, one can see from \cite[Corollary 2.10 (i)]{Yan97} that $(\R^1_{L/E}\G_m)(E)$ does not act trivially on $\phi_{i_0}$.
\end{remark}

\begin{lemma}\label{lem:gammanonzero}
Let $\beta$ be in $\Hom_{i_0(T_E)(F)}(\sigma^\epsilon,\one)$. If $\beta$ is nonzero, then $\beta(f_{i_0})$ is nonzero.
\end{lemma}
\begin{proof}
Using (\ref{eq:seesaw}), the proof of \cite[Proposition 3.8]{ourselves} or \cite[Proposition 3.10]{ourselves} shows that precomposition with the theta lift maps $\mathcal{S}(F^3)\ra\one\boxtimes\theta_{13}(\one)\cong\sigma^\epsilon$ and $\mathcal{S}(E)\ra\theta_E(\one)\boxtimes\one\cong\theta_E(\one)$ identify
\begin{align*}
\Hom_{i_0(T_E)(F)}(\sigma^\epsilon,\one)&\ra^\sim\Hom_{\U_1(F)\times(\R^1_{L/E}\G_m)(E)}(\mathcal{S}(F^3),\one)\\
&\cong\Hom_{\U_1(F)\times(\R^1_{L/E}\G_m)(E)}(\mathcal{S}(E),\one)\la^\sim\Hom_{\U_1(F)}(\theta_E(\one),\one).
\end{align*}
Now $\theta_E(\one)$ is isomorphic to $\one$ by \cite[Prop 3.4]{rogawski1992multiplicity}, and the image of $\phi_{i_0}$ in $\theta_E(\one)$ is nonzero by construction. Hence the desired result follows.
\end{proof}
Write $\beta_{i_0}$ for the unique element of $\Hom_{i_0(T_E)(F)}(\sigma^\epsilon,\one)$ satisfying $\beta_{i_0}(f_{i_0})=1$, which exists by Lemma \ref{lem:gammanonzero}. 

\begin{definition}\label{defn:beta}
For all $i=x\cdot i_0$ in the $G'(F)$-orbit of $i_0$, write $\beta_i$ for $\beta_{i_0}\circ x^{-1}$.
\end{definition}

 Definition \ref{defn:beta} 
 implies that $\beta_{g'\cdot i} = \beta_i\circ g'^{-1}$ for all $g'\in G'(F)$, and
 Lemma \ref{lemma:equiv_of_f_Xv} implies that $\beta_i(f_i)=1$.  Moreover, $\beta_i$ is compatible with the spherical vector $f_0$ as follows:

\begin{lemma}\label{lem:spherical_compare_Yang}
    Assume that $F$ is nonarchimedean, $K/F$ is unramified, $\chi$ is unramified, and $\psi$ has conductor~$0$. If $i$ in $\left\{E \hookrightarrow J\right\}$ arises from an isomorphism $K^3 \cong L_{\lambda_0}$ of Hermitian spaces for $K/F$ under which the image of $\O_K^3$ in $L_{\lambda_0}$ is $\O_L$-stable, then  $|\beta_i(f_0)| = 1 $.
\end{lemma}
\begin{proof}
If $i$ satisfies the above condition, then $g'\cdot i$ does too if and only if $g'$ lies in $G'(\O_F)i(T_E)(F)$. Therefore we can assume that $i=i_0$.

By inspecting the construction in \S\ref{ss:localPU3}, we see that $\phi_0$ in $\mathcal S(F^3)$ is a unitary spherical vector with respect to a self-dual $\O_F$-lattice in $W$ of the form $t \O_K^3$ for some $t$ in $K^\times$. 

On the other hand, by inspecting Yang's construction \cite{Yan97}, we see that $\phi_{i_0}$ in $\mathcal S(E)\cong\mathcal{S}(F^3)$ is a unitary spherical vector with respect to a self-dual (with respect to the symplectic form over $F$) $\O_L$-lattice in $L_\lambda$. Because any two such lattices are translates under $L^1$ and $\beta_{i_0}$ is $i_0(L^1)$-invariant, we can assume that $\phi_{i_0}$ is spherical with respect to the image of $t\O_K^3$. Then the desired result immediately follows.
\end{proof}

\subsubsection{Archimedean case}\label{sss:archYang}
In this subsubsection, assume that $F$ is archimedean. Because $G'(\RR)$ is compact, the existence of $i_0$ implies that $E\cong\RR^3$. Therefore $\R^1_{L/E}\G_m$ is anisotropic over $E$, so Remark \ref{rem:fieldcaseYang} applies. Write $\langle-,-\rangle_\sigma$ for the Hermitian pairing on $\sigma^-=\theta_{13}(\one)$ given by restricting the Hermitian pairing on $\mathcal{S}(F^3)$. Since $\phi_{i_0}$ is a unitary element of $\mathcal{S}(E)\cong\mathcal{S}(F^3)$ by construction \cite[p.~48]{Yan97}, we see that $f_{i_0}=\phi_{i_0}$ is unitary with respect to $\langle-,-\rangle_\sigma$. Because $G'(\RR)$ preserves $\langle-,-\rangle_\sigma$, this implies that $\beta_{i_0}$ equals $\langle-,f_{i_0}\rangle_\sigma$. More generally, this implies that $\beta_i$ equals $\langle-,f_i\rangle_\sigma$ for all $i$ in the $G'(\RR)$-orbit of $i_0$.

\subsection{Global torus periods}\label{ss:Yang}
In this subsection, assume that $F$ is a number field. We consider the following cuspidal automorphic representations of $G'$. Let $(\epsilon_v)_v$ be a sequence in $\{\pm1\}$ indexed by places $v$ of $F$ with
\begin{itemize}
\item $\epsilon_v=+1$ when $v$ splits in $K$ or $\chi_v^2=1$,
\item $\epsilon_v=-1$ when $v$ is archimedean,
\item $\prod_v\epsilon_v=\epsilon(\frac12,\chi^3)$.
\end{itemize}
Choose $\epsilon_1$ in $F^\times$ such that its image in $F^\times/\Nm_{K/F}(K^\times)$ is the unique element satisfying  
\begin{align*}
\epsilon_{1,v}=\omega_{K_v/F_v}(-1)\cdot\epsilon(\textstyle\frac12,\chi_v^3,\psi_v(\tr_{K_v/F_v}(\delta-)))\cdot\epsilon_v\in\{\pm1\}
\end{align*}
for every place $v$ of $F$, which exists by the Hasse principle and the fact that
\begin{align*}
\textstyle\prod_v\epsilon_{1,v}=(\prod_v\omega_{K_v/F_v}(-1))\cdot\epsilon(\frac12,\chi^3)\cdot(\prod_v\epsilon_v) = +1\cdot\epsilon(\frac12,\chi^3)^2 = +1.
\end{align*}
Then \cite[Proposition 3.3]{ourselves} and \S\ref{ss:seesaw} identify $\sigma\coloneqq\theta_{13}(\one)$ with $\bigotimes'_v\sigma_v^{\epsilon_v}$, where $\sigma_v^{\epsilon_v}$ is the irreducible smooth representation of $G'(F_v)$ from Definition \ref{defn:localsigma}. Write $\epsilon_v(E_v,\chi_v,\psi_v)$ for the sign from Definition~\ref{defn:cubicepsilon}.

We consider the following global torus periods. For any $i$ in $\{E\hookrightarrow J\}$, write $\mathcal{P}_i:\sigma\ra\C$ for the $\C$-linear map $f\mapsto\int_{[i(T_E)]}f(t')\dd{t'}$. Then \cite[Proposition 3.12]{ourselves} shows that $\mathcal{P}_i$ is nonzero only if, for every place $v$ of $F$, we have $\epsilon_v=\epsilon_v(E_v,\chi_v,\psi_v)$, and $i$ localizes to the unique $G'(F_v)$-orbit of $i_v$ in $\{E_v\hookrightarrow J_v\}$ satisfying
\begin{align*}
\Hom_{i_v(T_{E_v})(F_v)}(\sigma_v^{\epsilon_v},\one)\neq0.
\end{align*}
Assume these conditions hold for $\epsilon_v$ and $i$. Then $\R^1_{L_v/E_v}\G_m$ is anisotropic over $E_v$ for any archimedean place $v$ of $F$, so $\R^1_{L/E}\G_m$ is anisotropic over $E$. Moreover, since $J$ does not arise from a division algebra, there exists a unique $G'(F)$-orbit of such $i$ in $\{E\hookrightarrow J\}$ \cite[Lemma 15.5.(2)]{gan2021twisted}. Let $i$ be an element of this $G'(F)$-orbit. Write $f_{i_v}$ for the element of $\sigma^{\epsilon_v}_v$ from Definition \ref{defn:localYang}, and write $f_i$ for the element $\otimes'_vf_{i_v}$ of $\sigma$.

Let $S$ be a finite set of places of $F$ such that,
\begin{itemize}
    \item for all $v$ not in $S$, $v$ is nonarchimedean, $K_v/F_v$ is unramified, and $\chi_v$ is unramified,   
    \item $\O_{K,S}$ is a free $\O_{F,S}$-module.
\end{itemize}
Then $\R^1_{\O_{K,S}/\O_{F,S}}\G_m$ is a smooth model of $\R^1_{K/F}\G_m$ over $\O_{F,S}$, and its Lie algebra is the rank-$1$ free $\O_{F,S}$-module $(\O_{K,S})^{\tr=0}$. Hence there exists a nowhere vanishing, translation-invariant $1$-form $\mu$ on $\R^1_{\O_{K,S}/\O_{F,S}}\G_m$. 

Note that $(\R^1_{\O_{K,S}/\O_{F,S}}\G_m)_{\O_{E,S}}=\R^1_{\O_{L,S}/\O_{E,S}}\G_m$ is a smooth model of $\R^1_{L/E}\G_m$ over $\O_{E,S}$, and the pullback of $\mu$ to $\R^1_{\O_{L,S}/\O_{E,S}}\G_m$ remains nowhere vanishing and translation-invariant. For every place $w$ of~$E$, write $\vol_\mu$ for the associated measure on $L^1_w$, and write $M_w$ for the volume with respect to $\vol_\mu$ of the maximal compact subgroup of $L^1_w$. When $w$ does not lie above $S$, we have $M_w=L_w(1,\omega_{L_w/E_w})^{-1}$.
\begin{prop}\label{prop:Yang}
We have
    \begin{align*}
        \abs{\mathcal P_i (f_i)}^2 = L(\textstyle\frac12,\Ind^F_K\chi\otimes\Ind^F_E\one)\cdot\Delta_{\O_E/\Z}^{1/2}\cdot \displaystyle\prod_{v\in S} C_{v},
    \end{align*}
    where $\Delta_{\O_E/\Z}$ in $\Z$ denotes the discriminant of $\O_E/\Z$, and $C_{v}$ is a nonzero constant depending only on $\chi_v$, $E_v$, and $\mu$.
\end{prop}
By Proposition \ref{prop:Yang}, the product $\prod_{v\in S} C_v$ is independent of $\mu$.

\begin{proof}
Choose $\lambda$ in $E^\times$ and an isomorphism $K^3\cong L_\lambda$ of Hermitian spaces for $K/F$ that give rise to $i$ as in \S\ref{ss:unitarysetup}. Using (\ref{eq:seesaw}), for all $\phi$ in $\mathcal{S}(\A_F^3)\cong\mathcal{S}(\A_E)$ we see that
\begin{align}\label{eq:globalseesaw}
\begin{split}
\mathcal{P}_i(\theta_{13}(\phi,1)) = \int_{[i(T_E)]}\theta_{13}(\phi,1)(t')\dd{t'} &= \frac1{\vol([\U_1])}\int_{[\R^1_{L/E}\G_m]}\theta_{13}(\phi,1)(t')\dd{t'} \\
&= \frac1{\vol([\U_1])}\int_{[\U_1]}\theta_E(\phi,1)(u)\dd{u} = \theta_E(\phi,1)(1)
\end{split}
\end{align}
since $\theta_E(\one)$ is isomorphic to $\one$ by \cite[Prop 3.4]{rogawski1992multiplicity} and \cite[Proposition 1.2]{GRS93}.

We will apply (\ref{eq:globalseesaw}) as follows. For every place $v$ of $F$, write $\phi_{i_v}$ for the element of $\mathcal{S}(F_v^3)$ from \S\ref{ss:actuallocalvector}, and write $\phi_i$ for the element $\otimes'_v\phi_{i_v}$ of $\mathcal{S}(\A_F^3)$. By construction, $\phi_i$ is an $(\R^1_{L/E}\G_m)(\A_E)$-translate of the element of $\mathcal{S}(\A_F^3)\cong\mathcal{S}(\A_E)$ defined as $\phi_\one$ in \cite[p.~48]{Yan97}. Because the Hermitian pairing on $\mathcal{S}(\A_E)$ is $(\R^1_{L/E}\G_m)(\A_E)$-invariant, this implies that $\phi_i$ still satisfies \cite[(2.18)]{Yan97}, so \cite[Theorem 2.6]{Yan97} shows that
\begin{align*}
|\theta_E(\phi_i,1)(1)|^2 = L(\textstyle\frac12,\Ind^F_K\chi\otimes\Ind^F_E\one)\cdot\displaystyle\frac{\vol([\R^1_{L/E}\G_m])}{2^{\pi_0(\operatorname{Spec}{E})}L(1,\omega_{L/E})}\cdot\prod_wB_w,
\end{align*}
where $w$ runs over places of $E$, and
\begin{align*}
B_w\coloneqq\begin{cases}
(1+q_w^{-1})^{-1} & \mbox{when }w\mbox{ is inert in }L\mbox{ and }\chi_v\circ\Nm_{L_w/K_v}\mbox{ is ramified}, \\
q_w^{-c(\chi_v\circ\Nm_{L_w/K_v})}(1-q_w^{-1})^2 & \mbox{when }w\mbox{ is split in }L\mbox{ and }\chi_v\circ\Nm_{L_w/K_v}\mbox{ is ramified}, \\
1 & \mbox{otherwise.}
\end{cases}
\end{align*}
By \cite[Main Theorem]{Ono63}, we have
\begin{align*}
\vol([\R^1_{L/E}\G_m]) = 2^{\pi_0(\operatorname{Spec}{E})}L(1,\omega_{L/E})\cdot\Delta_{\O_E/\Z}^{1/2}\cdot\prod_wD_w,
\end{align*}
where
\begin{align*}
D_w\coloneqq\begin{cases}
M_w^{-1}L_w(1,\omega_{L_w/E_w})^{-1} & \mbox{when }w\mbox{ is nonarchimedean},\\
M_w^{-1} & \mbox{when }w\mbox{ is archimedean}.
\end{cases}
\end{align*}
Therefore we get
\begin{align*}
|\theta_E(\phi_i,1)(1)|^2 = L(\textstyle\frac12,\Ind^F_K\chi\otimes\Ind^F_E\one)\cdot\Delta_{\O_E/\Z}^{1/2}\cdot\displaystyle\prod_wB_wD_w.
\end{align*}
We conclude by noting that $f_i=\theta_{13}(\phi_i,1)$ and that $C_v\coloneqq\prod_{w\mid v}B_wD_w$ depends only on $\chi_v$, $E_v$, and $\mu$.
\end{proof}

\section{Formulas for Fourier coefficients}\label{s:formulas}
Our goal in this section is to prove (modulo calculating certain local integrals) our main results on the Fourier coefficients of quaternionic modular forms on $\mathsf{G}_2$ over totally real fields. We will calculate these local integrals later in \S\ref{s:unramified}, \S\ref{s:ramified}, and \S\ref{s:archimedean}.

We start in \S\ref{ss:qds} by recalling quaternionic discrete series representations on $\mathsf{G}_2$, which we use to define quaternionic modular forms and their Fourier coefficients in \S\ref{ss:qmf}. Our quaternionic modular forms of interest arise from an exceptional theta lift between $\PU_3$ and $\mathsf{G}_2$, using theta kernels on a quasi-split adjoint form of~$\mathsf{E}_6$. We gather basic facts about the Fourier coefficients of these theta functions in \S\ref{ss:E6}, \S\ref{ss:localfourierE6}, and \S\ref{ss:globalFourierE6}.

In \S\ref{ss:G2local}, we relate Fourier coefficients on $\mathsf{E}_6$ to Fourier coefficients on $\mathsf{G}_2$, leading to the definition of the local integrals in \S\ref{ss:locint}. Finally, in \S\ref{ss:main} we put everything together and prove our main results.

\subsection{Quaternionic discrete series on $\mathsf{G}_2$}\label{ss:qds}
We start with some notation on $\mathsf{G}_2$. Write $\g$ for the Lie algebra over $F$ defined as $\widetilde\g_F$ in \cite[\S2.2]{ourselves}, and write $G$ for its automorphism group over $F$. Recall that $G$ is the connected split simple group of type $\mathsf{G}_2$ over $F$ \cite[\S2.3]{ourselves}. Following \cite[\S2.5]{ourselves}, write $P$ for the Heisenberg parabolic subgroup of $G$, write $N$ for the unipotent radical of $P$, and write $M$ for the Levi subgroup of $P$. We can naturally identify $M$ with $\GL_2$.

Write $Z$ for the center of $N$, and recall that $N/Z$ is abelian \cite[\S2.5]{ourselves}. Write $\X$ for $\Hom_F(N/Z,F)$. We can naturally identify $\X$ and $N/Z$ with $F^4$ \cite[\S2.5]{ourselves}, and under this identification, the evaluation pairing $\X\times(N/Z)\ra F$ corresponds to $$\langle(a,b,c,d),(a',b',c',d')\rangle \coloneqq ad'-bc'+cb'-da'.$$
Write $q$ for the quartic form $(a,b,c,d)\mapsto b^2c^2+18abcd-4ac^3-4db^3-27a^2d^2$ on $\X$.

For the rest of this subsection, assume that $F$ is an archimedean local field. We now recall the \emph{quaternionic discrete series} on $G(\RR)$. Recall that the maximal compact subgroup $K$ of $G(\RR)$ is $\SU(2)_\ell\times^{\{\pm1\}}\SU(2)_s$, where the subscripts mean that their complexifications induce $\ell$ong and $s$hort root subgroups, respectively, of $G(\C)$ \cite[Proposition 4.1]{GW96}. For any positive integer $n$, write $\pi_n$ for the irreducible smooth representation of $G(\RR)$ defined as $\pi'_{2n+2}$ in \cite[Proposition 5.7]{GW96}. In particular, the minimal $K$-type of $\pi_n$ is $\mathbb{V}_n\coloneqq\Sym^{2n}\boxtimes\one$.

For all $\mathcal{E}$ in $\X(\RR)$, write $\psi_\mathcal{E}:N(\RR)\ra \C^1$ for the unitary character $\psi(\langle\mathcal{E},-\rangle)$. Write $r$ for the unique nonzero real number such that $\psi(x)=e^{-irx}$. Write $r_0(i)$ for the element $(1,-i,-1,i)$ of $(N/Z)(\C)$, and let $\mathcal{E}$ be an element of $\X(\RR)$ such that $\langle r\mathcal{E},m\cdot r_0(i)\rangle$ is nonzero for all $m$ in $M(\RR)$. For all integers $v$, write $K_v$ for the associated $K$-Bessel function, and write $\mathcal{W}^\mathcal{E}_{n,v}:M(\RR)\ra\C$ for the function given by
\begin{align*}
m\mapsto\left[\frac{|\langle r\mathcal{E},m\cdot r_0(i)\rangle|}{\langle r\mathcal{E},m\cdot r_0(i)\rangle}\right]^v\cdot\det(m)^n\cdot|\det(m)|\cdot K_v(|\langle r\mathcal{E},m\cdot r_0(i)\rangle|),
\end{align*}
which agrees with the expression in \cite[Theorem 1.2.1 (1) (b)]{MR4094735} by \cite[Proposition 2.3.1]{MR4094735}.

Write $\{x_\ell,y_\ell\}$ for the standard basis of the standard representation of $\SU(2)_\ell$, so that $\{x^{n+v}_\ell y_\ell^{n-v}\}_{\abs{v}\leq n}$ is a basis of $\mathbb{V}_n$. Then $\Hom_{N(\RR)}(\pi_n,\psi_\mathcal{E})=\Hom_{G(\RR)}(\pi_n,\Ind_{N(\RR)}^{G(\RR)}\psi_\mathcal{E})$ is spanned by the unique map $\mathcal{W}^\mathcal{E}_n$ such that, for all $\abs{v}\leq n$, the function $\mathcal{W}^\mathcal{E}_n(x^{n+v}_\ell y_\ell^{n-v}):G(\RR)\ra\C$ restricted to $M(\RR)$ equals
\begin{align*}
m\mapsto\frac1{(n+v)!(n-v)!}\mathcal{W}^\mathcal{E}_{n,-v}(m)
\end{align*}
\cite[Theorem 1.2.1 (1) (b)]{MR4094735}. On the other hand, if $\mathcal{E}$ is an element of $\X(\RR)$ such that $\langle r\mathcal{E},m\cdot r_0(i)\rangle$ vanishes for some $m$ in $M(\RR)$, then $\Hom_{N(\RR)}(\pi_n,\psi_\mathcal{E})$ vanishes \cite[Theorem 1.2.1 (1) (a)]{MR4094735}.

\subsection{Quaternionic modular forms on $G$}\label{ss:qmf}
In this subsection, assume that $F$ is a number field. Thanks to the results of Pollack \cite{MR4094735} recalled in \S\ref{ss:qds}, the following automorphic forms on $G$ have a good theory of Fourier coefficients.
\begin{definition}
We say that a cuspidal automorphic representation $\pi$ of $G(\A_F)$ is \emph{quaternionic} if, for all archimedean places $v$ of $F$, there exists a positive integer $n_v$ such that $\pi_v$ is isomorphic to the representation~$\pi_{n_v}$ from \S\ref{ss:qds}. We say that $(n_v)_{v\mid\infty}$ is the \emph{weight} of $\pi$.
\end{definition}
Let $\pi$ be a cuspidal automorphic representation of $G(\A_F)$ that is quaternionic of weight $(n_v)_{v\mid\infty}$. For any $\mathcal{E}$ in $\X(F)$, write $\psi_\mathcal{E}:[N/Z]\ra \C^1$ for the unitary character $\psi(\langle \mathcal{E},-\rangle)$. For any $\mathcal{F}^\infty$ in $\bigotimes'_{v\nmid\infty}\pi_v$, the function
\begin{align*}
\mathcal{F}_\infty\mapsto(\mathcal{F}^\infty\otimes\mathcal{F}_\infty)_{N,\psi_\mathcal{E}}(1)\coloneqq\int_{[N]}(\mathcal{F}^\infty\otimes\mathcal{F}_\infty)(n)\psi_{\mathcal{E}}(n)^{-1}\dd{n}
\end{align*}
yields an element of $\Hom_{N(F\otimes_{\mathbb{Q}}\RR)}(\bigotimes_{v\mid\infty}\pi_v,\bigotimes_{v\mid\infty}\psi_{v,\mathcal{E}})=\bigotimes_{v\mid\infty}\Hom_{N(F_v)}(\pi_v,\psi_{v,\mathcal{E}})$. If this space is nonzero, then \S\ref{ss:qds} shows that it is $1$-dimensional. Hence, after fixing an isomorphism $\pi_v\cong\pi_{n_v}$ for all archimedean places $v$ of $F$, there is a unique $a_\mathcal{E}(\mathcal{F}^\infty)$ in $\C$ such that
\begin{align*}
(\mathcal{F}^\infty\otimes\mathcal{F}_\infty)_{N,\psi_\mathcal{E}}(g_v) = a_\mathcal{E}(\mathcal{F}^\infty)\prod_{v\mid\infty}\mathcal{W}^{\mathcal{E}_v}_{n_v}(\mathcal{F}_v)(g_v)
\end{align*}
for all $\mathcal{F}_\infty = \otimes_{v\mid\infty}\mathcal{F}_v$ in $\bigotimes_{v\mid\infty}\pi_v$ and $g_\infty=(g_v)_{v\mid\infty}$ in $G(F\otimes_{\mathbb{Q}}\RR)$, where $\mathcal{E}_v$ denotes the image of $\mathcal{E}$ in $\X(F_v)$.
\begin{definition}\label{defn:Fourier}
We say that $a_{\mathcal{E}}(\mathcal{F}^\infty)$ is the \emph{$\mathcal{E}$-th Fourier coefficient} of $\mathcal{F}^\infty$.
\end{definition}

When $F=\mathbb{Q}$, Definition \ref{defn:Fourier} agrees with the Fourier coefficient defined in \cite[Corollary 1.2.3]{MR4094735}.

\subsection{A quasi-split adjoint form of $\mathsf{E}_6$}\label{ss:E6} We study our quaternionic modular forms on $G$ using an \emph{exceptional theta lift} between $G'$ and $G$ in a quasi-split adjoint form of $\mathsf{E}_6$; let us recall this in the next few subsections. Write $\widetilde\g$ for the Lie algebra over $F$ defined as $\widetilde\g_J$ in \cite[\S2.2]{ourselves}, and write $\widetilde{G}$ for its connected automorphism group over $F$. Recall that $\widetilde{G}$ is the quasi-split adjoint form of $\mathsf{E}_6$ with respect to $K$ over $F$ \cite[\S2.2]{ourselves}. Following \cite[\S2.5]{ourselves}, write $\widetilde{P}$ for the Heisenberg parabolic subgroup of $\widetilde{G}$, write $\widetilde{N}$ for the unipotent radical of $\widetilde{P}$, and write $\widetilde{M}$ for the Levi subgroup of $\widetilde{P}$.

Recall that there is a natural injective morphism $G\times G'\hookrightarrow\widetilde{G}$ of groups over $F$ \cite[\S2.3]{ourselves}. Moreover, we have
\begin{align*}
P=\widetilde{P}\cap G,\quad M=\widetilde{M}\cap G, \quad N=\widetilde{N}\cap G, \quad\mbox{and}\quad G'\subseteq\widetilde{M}
\end{align*}
\cite[\S2.5]{ourselves}. The center of $\widetilde{N}$ equals $Z$, and the quotient $\widetilde{N}/Z$ is abelian; write $\widetilde\X$ for $\Hom_F(\widetilde{N}/Z,F)$. We can naturally identify $\widetilde{N}/Z$ and $\widetilde{\X}$ with $F\times J\times J\times F$ \cite[\S2.5]{ourselves}, and under this identification, the evaluation pairing $\widetilde{\X}\times(\widetilde{N}/Z)\ra F$ corresponds to
\begin{align}\label{eq:our_symplectic_form}
\langle(a,x,y,d),(a',x',y',d')\rangle \coloneqq ad'-\tr(x\circ y')+\tr(y\circ x')-da'.
\end{align}

By \cite[Lemma 4.3.1]{MR4094735}, the action of $\widetilde{M}$ on $\widetilde\X$ identifies $\widetilde{M}$ with a certain similitude group\footnote{In \cite{MR4094735}, this similitude group is denoted by $H_J$.} over $F$. Under this identification, write $\nu:\widetilde{M}\ra\G_m$ for the similitude character, and write $\widetilde{M}^1$ for the kernel of $\nu$.

Write $\O_{\min}$ for the $\widetilde{M}$-orbit of $(1,0,0,0)$ in $\widetilde\X$ over $F$, often called the {\em minimal orbit}.  Recall \cite[Proposition 8.1]{gan2021twisted} that $\O_{\min}$ equals
\begin{align*}
\big\{0\neq(a,x,y,d)\in\widetilde\X\mid x^\#=ay,\,y^\#=dx,\mbox{ and }l(x)\circ l^*(y)=ad\mbox{ for all }l\mbox{ in }L_J(F)\big\},
\end{align*}
where $L_J\subseteq\GL_J$ is the subgroup of linear maps preserving $N_J$, and $(-)^*$ denotes the dual with respect to the trace pairing $(X,Y)\mapsto\tr(X\circ Y)$ on $J$.
\begin{lemma}\label{lem:transitive_O_min}
The group $\widetilde M^1(F)$ acts transitively on $\O_{\min}(F)$.
\end{lemma}
\begin{proof}
Write $S$ for the $\widetilde M^1(F)$-orbit of $(1,0,0,0)$, which we want to show equals $\O_{\min}(F)$. For any $t$ in $F^\times$, let $g_t$ be a linear automorphism of $J$ over $F$ such that $\det(g_t(z)) = t^2 \det z$ for all $z$ in $J$ (e.g. we can take $g_t$ to be the operation of multiplying the first row and first column of a Hermitian matrix by $t$). Then \cite[p.~1221]{MR4094735} yields an element $M(t,g_t)$ of $\widetilde{M}^1(F)$ such that
\begin{align*}
M(t, g_t) (a, x, y, d) = (t^{-1} a, t^{-1} g_t(x), tg_t^*(y), td).
\end{align*}
In particular, $(t, 0, 0, 0)$ lies in $S$.

Next, for any $Z$ in $J$, \cite[p.~1221]{MR4094735} also yields an element $n(Z)$ of $\widetilde M^1(F)$ satisfying
\begin{align*}
n(Z) (t, 0, 0,0) = (t, tZ, tZ^\#, t\det Z).
\end{align*}
Therefore $S$ contains every $(a,x,y,d)$ in $\O_{\min}(F)$ with $a\neq0$. On the other hand, given an $\mathcal X$ in $\O_{\min}(F)$, the locus $\{\widetilde{m}\in\widetilde{M}^1\mid \widetilde{m}\cdot\mathcal{X}=(a,x,y,d)\mbox{ with }a\neq0\}$ is a dense open subvariety of $\widetilde{M}^1$ over $F$. Since $\widetilde{M}^1$ is $F$-unirational and $F$ is infinite, this implies that this subvariety contains an $F$-point. Hence $\widetilde M^1(F) \cdot \mathcal X$ meets $S$, so $\mathcal X$ lies in $S$. 
\end{proof}

\subsection{Local Fourier coefficients of $\widetilde{G}$}\label{ss:localfourierE6}
In this subsection, assume that $F$ is a local field. Write $\Omega$ for the minimal representation of $\widetilde{G}(F)$ in the sense of \cite[Definition 3.6]{gan2005minimal} or \cite[Definition 4.6]{gan2005minimal}, which is an irreducible smooth representation of $\widetilde{G}(F)$. When $K/F$ is unramified, $\widetilde{G}$ is unramified over $F$, so $\Omega$ is unramified \cite[Corollary 7.4]{gan2005minimal}. We will use $\Omega$ as a kernel for our exceptional theta lift.

For all $\mathcal{X}$ in $\widetilde{\X}(F)$, write $\psi_{\mathcal{X}}:\widetilde{N}(F)\ra \C^1$ for the unitary character given by $\psi(\langle\mathcal{X},-\rangle)$. Then $(\widetilde{N}(F),\psi_{\mathcal{X}})$-equivariant functionals on $\Omega$ are the local analog of Fourier coefficients for $\Omega$; we normalize them as follows.

\subsubsection{Nonarchimedean case}\label{sss:nonarchE6}
In this subsubsection, assume that $F$ is nonarchimedean. Then there exists a $\widetilde{P}(F)$-equivariant short exact sequence
\begin{align*}
0 \to C_c^\infty(\O_{\min}(F)) \to \Omega_{Z(F)} \to \Omega_{\widetilde N(F)} \to 0
\end{align*}
\cite[Theorem 6.1]{magaard1997exceptional}\footnote{Note that \cite{magaard1997exceptional} works in the split case. However, the proof of \cite[Theorem 6.1]{magaard1997exceptional} holds verbatim in the quasi-split case if one uses \cite[Proposition 11.5 (i)]{gan2005minimal} and \cite[Proposition 11.7 (ii)]{gan2005minimal} instead of \cite[Lemma 6.2]{magaard1997exceptional}.}, where the $\widetilde{P}(F)$-action on $C_c^\infty(\O_{\min}(F))$ is given by                           
\begin{itemize}
    \item $(\widetilde{m}\varphi)(\mathcal X) = |\nu(\widetilde{m})|^{1/5} \varphi(\widetilde{m}^{-1} \mathcal X)$ for $\widetilde{m}$ in $\widetilde M(F)$,
    \item $(\widetilde{n}\varphi)(\mathcal X) = \psi_{\mathcal X} (\widetilde{n}) \varphi(\mathcal X)$ for $\widetilde{n}$ in $\widetilde N(F)$.
\end{itemize}

In particular, for all $\mathcal X$ in $\O_{\min}(F)$ we have
\begin{equation}\label{eqn:twisted_coinvariants_min_orbit}
    C_c^\infty(\O_{\min}(F))_{\widetilde N(F), \psi_{\mathcal X}} \ra^\sim \Omega_{\widetilde N(F), \psi_{\mathcal X}}.
\end{equation}
Moreover, the spaces in (\ref{eqn:twisted_coinvariants_min_orbit}) are 1-dimensional \cite[Lemma 6.2]{magaard1997exceptional}.

\begin{definition}\label{defn:alphanonarch}
Write $\alpha_{\mathcal{X}}$ for the functional in $\Hom_{\widetilde{N}(F)}(\Omega,\psi_{\mathcal{X}})$ identified via \eqref{eqn:twisted_coinvariants_min_orbit} with $\varphi\mapsto\varphi(\mathcal{X})$.
\end{definition}
\begin{lemma}\label{lem:alpha_equivariance}
For all $\widetilde{h}$ in $\widetilde{M}(F)$, $\mathcal{X}$ in $\O_{\min}(F)$, and $\varphi$ in $\Omega$, we have
$$\alpha_{\mathcal X} (\varphi) = |\nu(\widetilde{h})|^{-1/5}\alpha_{\widetilde{h}\cdot\mathcal X}(\widetilde{h}\cdot\varphi).$$
\end{lemma}
\begin{proof}
This follows immediately from the above description of the $\widetilde{M}(F)$-action on $C_c^\infty(\O_{\min}(F))$.
\end{proof}

\subsubsection{Archimedean case}\label{sss:localFourierE6arch}
In this subsubsection, assume that $F$ is archimedean. Recall that the maximal compact subgroup $\widetilde{K}$ of $\widetilde{G}(\RR)$ is isomorphic to $\SU(2)_\ell\times^{\{\pm1\}}\SU(6)/\mu_3(\C)$ \cite[Proposition 4.1]{GW96}, and recall from \cite[Proposition 12.11]{gan2005minimal} that $\Omega$ is isomorphic to the representation defined as $\pi'_4$ in \cite[Proposition 5.7]{GW96}. In particular, the minimal $\widetilde{K}$-type of $\Omega$ is $\widetilde{\mathbb{V}}_1\coloneqq\Sym^2\boxtimes\one$.

Let $\mathcal{X}$ be in $\O_{\min}(\RR)$, and recall from \S\ref{ss:qds} the nonzero real number $r$. Write $\widetilde{r}_0(i)$ for the element $(1,-i,-1,i)$ of $(\widetilde{N}/Z)(\C)$. For all integers $v$, write $\widetilde{\mathcal{W}}^{\mathcal{X}}_{v}:\widetilde{M}(\RR)\ra\C$ for the function
\begin{align}\label{eq:Pollackformula}
\widetilde{m}\mapsto\left[\frac{|\langle r\mathcal{X},\widetilde{m}\cdot \widetilde{r}_0(i)\rangle|}{\langle r\mathcal{X},\widetilde{m}\cdot \widetilde{r}_0(i)\rangle}\right]^v\cdot\nu(\widetilde{m})\cdot|\nu(\widetilde{m})|\cdot K_v(|\langle r\mathcal{X},\widetilde{m}\cdot\widetilde{r}_0(i)\rangle|),
\end{align}
which agrees with the expression in \cite[Theorem 1.2.1 (1) (b)]{MR4094735} by \cite[Proposition 2.3.1]{MR4094735}. (Because $\mathcal X$ lies in  $\O_{\min}(\RR)$, we have  $\langle r\mathcal{X},\widetilde{m}\cdot \widetilde{r}_0(i)\rangle \neq 0$ for all $\widetilde m$ in $\widetilde M(\RR)$.)

Recall from \S\ref{ss:qds} the standard basis $\{x_\ell,y_\ell\}$ of the standard representation of $\SU(2)_\ell$, so that $\{x^2_\ell,x_\ell y_\ell,y^2_\ell\}$ is a basis of $\widetilde{\mathbb{V}}_1$. Recall that $\Hom_{\widetilde{N}(\RR)}(\Omega,\psi_{\mathcal{X}})$ is $1$-dimensional by \cite[Theorem 1.2.1 (1) (b)]{MR4094735}.

\begin{definition}\label{defn:archalpha}
Write $\alpha_{\mathcal{X}}$ for the unique map in $\Hom_{\widetilde{N}(\RR)}(\Omega,\psi_{\mathcal{X}})=\Hom_{\widetilde{G}(\RR)}(\Omega,\Ind_{\widetilde{N}(\RR)}^{\widetilde{G}(\RR)}\psi_{\mathcal{X}})$ such that, for all $\abs{v}\leq1$, the function $\alpha_{\mathcal{X}}(x^{1+v}_\ell y^{1-v}_\ell):\widetilde{G}(\RR)\ra\C$ restricted to $\widetilde{M}(\RR)$ equals
\begin{align*}
\widetilde{m}\mapsto\frac1{(1+v)!(1-v)!}\widetilde{\mathcal{W}}^{\mathcal{X}}_{-v}(\widetilde{m}),
\end{align*}
which exists and is well-defined by \cite[Theorem 1.2.1 (1) (b)]{MR4094735}.
\end{definition} 

\begin{lemma}\label{lem:alpha_equivariancearch}
For all $\widetilde{h}$ in $\widetilde{M}^1(\RR)$, $\mathcal{X}$ in $\O_{\min}(\RR)$, and $\varphi$ in $\Omega$, we have
$$\alpha_{\mathcal X} (\varphi) = \alpha_{\widetilde{h}\cdot\mathcal X}(\widetilde{h}\cdot\varphi).$$
\end{lemma}
\begin{proof}
This follows immediately from $\nu(\widetilde{h})=1$ and (\ref{eq:Pollackformula}).\end{proof}

\subsection{Global Fourier coefficients of $\widetilde{G}$}\label{ss:globalFourierE6}
In this subsection, assume that $F$ is a number field. For every place $v$ of $F$, write $\Omega_v$ for the minimal representation of $\widetilde{G}(F_v)$ from \S\ref{ss:localfourierE6}, and for every $\mathcal{X}_v$ in $\O_{\min}(F_v$), write $\alpha_{\mathcal{X}_v}:\Omega_v\ra\psi_{v,\mathcal{X}_v}$ for the $\widetilde{N}(F_v)$-equivariant functional from \S\ref{ss:localfourierE6}.

When $F_v$ is nonarchimedean, $K_v/F_v$ is unramified, $\chi_v$ is unramified, and $\psi_v$ has conductor $0$, write $\varphi_{0,v}$ for the nonzero $\widetilde{G}(\O_{F_v})$-fixed element of $\Omega_v$ from \S\ref{ss:sphericalvectorE6} below. In particular, Corollary \ref{cor:calculate_spherical_easy} below implies that, for all $\mathcal{X}$ in $\O_{\min}(F)$, we have $\alpha_{\mathcal{X}_v}(\varphi_{0,v})=1$ for cofinitely many $v$, where $\mathcal{X}_v$ denotes the image of $\mathcal{X}$ in $\O_{\min}(F_v)$.

Write $\Omega$ for $\bigotimes'_v\Omega_v$. Recall that residues of Eisenstein series yield a $\widetilde{G}(\A_F)$-equivariant embedding \cite[\S14.3]{gan2021twisted}
\begin{align*}
\theta:\Omega\hookrightarrow L_{\disc}^2([\widetilde{G}]).
\end{align*}
The (global) Fourier coefficients for $\Omega$ take the following particularly simple form.
\begin{lemma}\label{lemma:c_X(Omega)=1}
After replacing $\theta$ with a $\C^\times$-multiple, the following is true: for all $\mathcal{X}$ in $\O_{\min}(F)$ and $\varphi=\otimes'_v\varphi_v$ in $\Omega$ with $\varphi_v=\varphi_{0,v}$ for cofinitely many $v$, we have
    $$\theta(\varphi)_{\widetilde{N},\psi_{\mathcal{X}}}(1) \coloneqq \int_{[\widetilde N]} \theta(\varphi) (\widetilde{n}) \psi_{\mathcal X}(\widetilde{n})^{-1}\dd{\widetilde{n}} = \prod_v \alpha_{\mathcal{X}_v} (\varphi_v).$$
\end{lemma}
\begin{proof}
For all $\mathcal{X}$ in $\O_{\min}(F)$ and every place $v$ of $F$, the space $\Hom_{\widetilde{N}(F_v)}(\Omega_v,\psi_{v,\mathcal{X}_v})$ is $1$-dimensional. Hence there is a unique $c_{\mathcal{X}}$ in $\C$ such that, for all $\varphi=\otimes'_v\varphi_v$ in $\Omega$ with $\varphi_v=\varphi_{0,v}$ for cofinitely many $v$, we have
    $$\int_{[\widetilde N]} \theta(\varphi)(\widetilde{n}) \psi_{\mathcal X}(\widetilde{n})^{-1}\dd{\widetilde{n}} = c_{\mathcal X}\prod_v \alpha_{\mathcal{X}_v}(\varphi_v).$$
    We claim that $c_{\widetilde{m}\cdot\mathcal X} = c_{\mathcal X}$ for all $\widetilde{m}$ in $\widetilde M^1(F)$. Indeed, we have 
\begin{align*}
c_{\widetilde{m}\cdot \mathcal X}\prod_v \alpha_{\mathcal X_v}(\varphi_v) &= c_{\widetilde{m}\cdot \mathcal X}\prod_v \alpha_{\widetilde{m}\cdot\mathcal X_v}(\widetilde{m}\cdot\varphi_v) & \mbox{by Lemma \ref{lem:alpha_equivariance} and Lemma \ref{lem:alpha_equivariancearch}}\\
&= \int_{[\widetilde{N}]}\theta(\widetilde{m}\cdot\varphi)(\widetilde{n})\psi_{\widetilde{m}\cdot\mathcal{X}}(\widetilde{n})^{-1}\dd{\widetilde{n}} \\
&= \int_{[\widetilde{N}]}\theta(\varphi)(\widetilde{n}\widetilde{m})\psi_{\mathcal{X}}(\widetilde{m}^{-1}\widetilde{n}\widetilde{m})^{-1}\dd{\widetilde{n}} & \mbox{since $\langle-,-\rangle$ is $\widetilde{M}$-invariant} \\
&= \int_{[\widetilde{N}]}\theta(\varphi)(\widetilde{m}\widetilde{n})\psi_{\mathcal{X}}(\widetilde{n})^{-1}\dd{\widetilde{n}} & \mbox{by $\widetilde{n}\mapsto\widetilde{m}\widetilde{n}\widetilde{m}^{-1}$} \\
&= \int_{[\widetilde{N}]}\theta(\varphi)(\widetilde{n})\psi_{\mathcal{X}}(\widetilde{n})^{-1}\dd{\widetilde{n}} \\ 
&= c_{\mathcal X}\prod_v \alpha_{\mathcal X_v}(\varphi_v).
\end{align*}
Therefore the desired result follows from Lemma \ref{lem:transitive_O_min}.
\end{proof}
Henceforth, we replace $\theta:\Omega\hookrightarrow L_{\disc}^2([\widetilde{G}])$ with a $\C^\times$-multiple such that the conclusion of Lemma \ref{lemma:c_X(Omega)=1} holds.

\subsection{Cubic algebras and $\O_{\min}$}\label{ss:G2local}
Fourier coefficients of $G$ are indexed by \emph{cubic algebras} as follows. Since $G$ is split over $F$, it and its various subgroups that we consider have natural models over $\Z$. Let $R$ be a ring; we say that an $R$-algebra $A$ is \emph{cubic} if $A$ is isomorphic to $R^3$ as an $R$-module. Recall that a \emph{good basis} of~$A$ is an element $(\alpha,\beta)$ of $A^2$ such that $\{1,\alpha,\beta\}$ is an $R$-basis of $A$ and $\alpha\beta$ lies in $R$. Then the proof of \cite[Proposition 3.1]{GGS02}\footnote{While \cite{GGS02} works over $R=\Z$, the proof holds verbatim over general $R$.} shows there exists $(a,b,c,d)$ in $R^4$ satisfying
\begin{align}\label{eq:goodbasis}
\alpha^2 = -ac+b\alpha-a\beta,\quad\beta^2=-bd+d\alpha-c\beta,\quad\alpha\beta=-ad,
\end{align}
and this induces a bijection
\begin{align}
\big\{\mbox{cubic }R\mbox{-algebras equipped with a good basis }(\alpha,\beta)\big\}\ra^\sim\X(R) \label{eqn:cubic_R-algebras}
\end{align}
that descends into an equivalence of groupoids $\{\mbox{cubic }R\mbox{-algebras}\}\ra^\sim M(R)\backslash\X(R)$ \cite[Proposition 3.1]{GGS02}. Under this correspondence, $q(a,b,c,d)$ in $R$ is a representative of the discriminant $\Delta_{A/R}$ in $R/(R^\times)^2$ \cite[p.~116]{GGS02}. 

Fourier coefficients of $\widetilde{G}$ can also be described in terms of cubic algebras as follows. Write $p:\widetilde{\X}\twoheadrightarrow\X$ for the map induced by $N/Z\hookrightarrow\widetilde{N}/Z$. Under our identifications, $p$ corresponds to the map
\begin{align*}
\id\times\tr\times\tr\times\id:F\times J\times J\times F\ra F\times F\times F\times F
\end{align*}
\cite[Lemma 2.3]{ourselves}. Recall from \S\ref{ss:unitarysetup} the set of $F$-algebra embeddings $\{E\hookrightarrow J\}$.

Assume that $R$ is a subring of $F$. Let $J(R)$ be an $R$-submodule of $J$ such that
\begin{itemize}
    \item $J(R)\otimes_RF$ equals $J$,
    \item $J(R)$ contains $1$,
    \item $J(R)$ is closed under $(-)^\#$,
    \item the image of $J(R)$ under $\tr: J\ra F$ lies in $R$.
\end{itemize}
\begin{exmp}\label{exmp}
When $F$ is a nonarchimedean local field, we take $J(\O_F)$ to be the set of Hermitian matrices in $\operatorname{M}_3(\O_K)$.
\end{exmp}
Write $\widetilde{\X}(R)$ for the image of $R\times J(R)\times J(R)\times R$ under our identification $F\times J\times J\times F=\widetilde{\X}(F)$.

Let $\mathcal{E}=(a,b,c,d)$ be an element of $\X(R)$, and write $A$ for the associated cubic $R$-algebra equipped with a good basis $(\alpha,\beta)$. Assume that $A\otimes_RF$ is isomorphic to a cubic \'etale $F$-algebra $E$. Write $\{A\hookrightarrow J(R)\}$ for the set of $R$-module embeddings $i:A\hookrightarrow J(R)$ such that $i_F:E\hookrightarrow J$ is an $F$-algebra embedding.

We have the following (integral, non-monic) generalization of \cite[Lemma 6.1]{ourselves}:
\begin{lemma}\label{lem:integralminorbitfibers}
We have a natural bijection
\begin{align*}
\{A\hookrightarrow J(R)\}\ra^\sim\widetilde{\X}(R)\cap\O_{\min}(F)\cap p^{-1}(\mathcal{E})
\end{align*}
given by $i\mapsto(a,i(\alpha),-i(\beta),d)$.
\end{lemma}
\begin{proof}
View $E$ as a Freudenthal algebra over $F$ in the sense of \cite[\S37.C]{KMRT}, and write $(-)^\#:E\ra E$ for the adjoint in the sense of \cite[\S38]{KMRT}. Then any $i$ in $\{E\hookrightarrow J\}$ is an embedding of Freudenthal algebras over $F$, so (\ref{eq:goodbasis}) implies that
\begin{align*}
\tr(i(\alpha))=\tr(\alpha)=b,\quad\tr(-i(\beta))=-\tr(\beta)=c,\quad i(\alpha)^\#=i(\alpha^\#)=-ai(\beta),\quad (-i(\beta))^\#=i(\beta^\#)=di(\alpha).
\end{align*}
This shows that $(a,i(\alpha),-i(\beta),d)$ lies in $p^{-1}(\mathcal{E})$. To see that $(a,i(\alpha),-i(\beta),d)$ lies in $\O_{\min}(F)$, first note that (\ref{eq:goodbasis}) implies that $i(\alpha)\circ i(\beta)=-ad$. Next, recall from \cite[p.~330]{bakic2021howe} that the group $L_J(F)$ is isomorphic to
\begin{align*}
\{g\in\GL_3(K)\mid\det g\mbox{ lies in }K^1\}\rtimes\Z/2,
\end{align*}
where the generator of $\Z/2$ acts via $x\mapsto\overline{x}$, and $g$ acts on $J$ via $x\mapsto gx\prescript{t}{}{\overline{g}}$. Since $ad$ lies in $F$, we see that $\overline{i(\alpha)}\circ\overline{i(\beta)}=-ad$. As for $g$, we have
\begin{align*}
g(i(\alpha))\circ g^*(i(\beta)) = \frac12\big[gi(\alpha)i(\beta)g^{-1}+\prescript{t}{}{\overline{g}}^{-1}i(\beta)i(\alpha)\prescript{t}{}{\overline{g}}\big].
\end{align*}
When $F$ is algebraically closed, $i:E\cong F^3\hookrightarrow J\cong\operatorname{M}_3(F)$ is conjugate to the diagonal embedding. Therefore, in general we have $i(z)\circ i(z')=i(z)i(z')$ for all $z$ and $z'$ in $J$, so $i(\alpha)\circ i(\beta)=-ad$ implies that the above expression also equals $-ad$. Altogether, this shows that $(a,i(\alpha),-i(\beta),d)$ lies in $\widetilde{\X}(R)\cap\O_{\min}(F)\cap p^{-1}(\mathcal{E})$.

Conversely, let $(a,x,y,d)$ be an element of $\widetilde{\X}(R)\cap\O_{\min}(F)\cap p^{-1}(\mathcal{E})$. Because $z^\#=z^2-\tr(z)z+\tr (z^\#)$ for all $z$ in $J$, we have
\begin{align*}
x^2&=-\tr(x^\#)+\tr(x)x+x^\# = -a\tr(y)+bx+ay = -ac+bx-a(-y), \\
(-y)^2&=-\tr(y^\#)+y^\#+\tr(y)y = -d\tr(x)+dx+cy = -bd+dx-c(-y),\\
x\circ(-y)&=-ad.
\end{align*}
Hence (\ref{eq:goodbasis}) implies that the unique $R$-module morphism $i:A\ra J(R)$ with $i(1)=1$, $i(\alpha)=x$, and $i(\beta)=-y$ becomes an $F$-algebra embedding $E\hookrightarrow J$ after applying $-\otimes_RF$. In particular, $i$ lies in $\{A\hookrightarrow J(R)\}$.
\end{proof}

\subsection{Local integrals}\label{ss:locint}
In this subsection, assume that $F$ is a local field. We now define the local integrals that arise when calculating the (global) Fourier coefficients of our quaternionic modular forms on $G$. Recall from \S\ref{ss:seesaw} the sign $\epsilon$ and the irreducible smooth representation $\sigma^{\epsilon}$ of $G'(F)$, recall from \S\ref{ss:G'localvector} the sign $\epsilon(E,\chi,\psi)$, and assume that $\epsilon=\epsilon(E,\chi,\psi)$. Then Proposition \ref{prop:epsilon_and_i} shows there exists a unique
$G'(F)$-orbit of $i$ in $\{E\hookrightarrow J\}$ satisfying $\Hom_{i(T_E)(F)}(\sigma^{\epsilon},\one)\neq0$. Let $i$ be an element of this $G'(F)$-orbit, write $\mathcal{X}$ for the corresponding element of $\O_{\min}(F)\cap p^{-1}(\mathcal{E})$ under Lemma \ref{lem:integralminorbitfibers}, and write $\beta_{\mathcal{X}}:\sigma^\epsilon\ra\one$ for the associated $i(T_E)(F)$-invariant functional denoted by $\beta_i$ in Definition \ref{defn:beta}.

\begin{definition}\label{defn:localzetaintegral}
For any $\varphi$ in $\Omega$ and $f$ in $\sigma^{\epsilon}$, write
\begin{align*}
\mathcal{I}(\mathcal{E},\varphi,f)\coloneqq\int_{i(T_E)(F)\backslash G'(F)}\alpha_{\mathcal{X}}(g'\cdot\varphi)\overline{\beta_{\mathcal{X}}(g'\cdot f)}\dd{g'}.
\end{align*}
\end{definition}
Because $G'$ is semisimple, it lies in $\widetilde{M}^1$. Hence Definition \ref{defn:beta}, Lemma \ref{lem:alpha_equivariance}, and Lemma \ref{lem:alpha_equivariancearch} imply that $\mathcal{I}(\mathcal{E},\varphi,f)$ does not depend on the choice of $\mathcal{X}$.

When $F$ is archimedean, write $N$ for the unique odd integer such that $\chi(z)=(z/\abs{z})^N$, and write $n$ for the positive integer $\frac{|N|+1}2$.

We will compute $\mathcal{I}(\mathcal{E},-,-)$ in the unramified case in \S\ref{s:unramified} and in the archimedean case in \S\ref{s:archimedean}:
\begin{prop}\label{prop:unram_and_arch_int}\hfill
\begin{enumerate}
    \item Assume that $F$ is nonarchimedean, $K/F$ and $\chi$ are unramified, and $\psi$ has conductor $0$. Write $\varphi_0$ for the nonzero $\widetilde{G}(\O_F)$-fixed element of $\Omega$ from \S\ref{ss:sphericalvectorE6} below, and recall from \S\ref{ss:localPU3} the nonzero $G'(\O_F)$-fixed element $f_0$ of $\sigma^+$. Then we have
    \begin{align*}
|\mathcal{I}(\mathcal{E},\varphi_0,f_0)|=\begin{cases}
1 & \mbox{if }\mathcal{E}\mbox{ lies in }\X(\O_F)\mbox{ and }\O_{\mathcal{E}}\cong\O_E,\\
0 & \mbox{if }\mathcal{E}\mbox{ does not lie in }\X(\O_F),
\end{cases}
\end{align*}
where $\O_{\mathcal{E}}$ denotes the cubic $\O_F$-algebra associated with $\mathcal{E}$ in $\X(\O_F)$ via~\eqref{eqn:cubic_R-algebras}.

    \item Assume that $F$ is archimedean. Write $\varphi_0$ for the element of $\Omega$ denoted by $\varphi_N$ in \S\ref{ss:Gtildearch} below, and write $f_0$ for the element of $\sigma^-$ from \S\ref{ss:archlocalintegral} below. After replacing $f_0$ or $\varphi_0$ with a $\C^\times$-multiple independent of $\mathcal{E}$ and $\mathcal{X}$, we have
    \begin{align*}
        \mathcal{I}(\mathcal{E},\varphi_0,f_0) = q(\mathcal{E})^{(n-1)/2}\left[\frac{|r(ai+b-ci-d)|}{r(ai+b-ci-d)}\right]^{-n}K_{-n}(|r(ai+b-ci-d)|).
    \end{align*}
\end{enumerate}
\end{prop}
\begin{proof}
Part (1) is Theorem \ref{prop:unramified_integral}, and part (2) follows from Theorem \ref{thm:arch_integral}.
\end{proof}

\subsection{Global Fourier coefficients of $G$}\label{ss:main}
In this subsection, assume that $F$ is a number field. We use an exceptional theta lift to construct our quaternionic modular forms on $G$ as follows. Recall from \S\ref{ss:Yang} the sequence $(\epsilon_v)_v$ and the associated cuspidal automorphic representation $\sigma$ of $G'(\A_F)$. For all archimedean places $v$ of $F$, write $r_v$ for the nonzero real number associated with $\psi_v$ from \S\ref{ss:qds}, write $N_v$ for the odd integer associated with $\chi_v$ from \S\ref{ss:locint}, and write $n_v$ for the positive integer $\frac{|N_v|+1}2$.

Write $\theta(-)$ for the theta lift from $G'$ to $G$ from \cite[Definition 2.8]{ourselves}. Assume that $K_v/F_v$ is unramified for every place $v$ of $F$ above $2$, and assume that $L(\frac12,\chi)\neq0$. Then \cite[Theorem B]{ourselves} shows that $\pi\coloneqq\theta(\sigma)$ is a cuspidal automorphic representation of $G(\A_F)$ that is quaternionic of weight $(n_v)_{v\mid\infty}$.

We actually work \emph{without} the unramified assumption at 2, as follows. In this generality, our arguments from \cite{ourselves} still yield a $G(\A_F)$-equivariant map  (a priori possibly zero) $\pi \coloneqq \bigotimes'_v\theta(\sigma_v^{\epsilon_v})\ra L^2_{\cusp}([G])$, where $\theta(\sigma_v^{\epsilon_v})$ is isomorphic to $\pi_{n_v}$ for all archimedean places $v$ of $F$. After fixing such isomorphisms, Definition \ref{defn:Fourier} still goes through, where now $\mathcal F^\infty$ lies in $\bigotimes'_{v\nmid\infty} \theta(\sigma_v^{\epsilon_v})$. For example, if the map $\pi\ra L^2_{\cusp}([G])$ is zero, then $a_{\mathcal E} (\mathcal F^\infty)$ vanishes for all $\mathcal{E}$ in $\X(F)$.

We will calculate the Fourier coefficients of certain elements of $\pi$. Let $S$ be a finite set of places of $F$ where
\begin{itemize}
\item for all $v$ not in $S$, $v$ is nonarchimedean, $K_v/F_v$ and $\chi_v$ are unramified, and $\psi_v$ has conductor $0$,
\item $\O_{K,S}$ is a free $\O_{F,S}$-module.
\end{itemize}
Recall from \S\ref{ss:Yang} the nowhere vanishing, translation-invariant 1-form $\mu$ on $\R^1_{\O_{K,S}/\O_{F,S}} \mathbb G_m$. For all $v$ not in~$S$, write $\varphi_{0,v}$ for the nonzero $\widetilde{G}(\O_{F_v})$-fixed element of $\Omega_v$ from \S\ref{ss:sphericalvectorE6} below, and write $f_{0,v}$ for the nonzero $G'(\O_{F_v})$-fixed element of $\sigma_v^+$ from \S\ref{ss:localPU3}. For all archimedean places $v$ of $F$, write $\varphi_{0,v}$ and $f_{0,v}$ for the elements of $\Omega_v$ and $\sigma_v^-$, respectively, from Proposition \ref{prop:unram_and_arch_int}.(2).

Let $\varphi=\otimes'_v\varphi_v$ in $\Omega$ and $f=\otimes'_vf_v$ in $\sigma$ be elements such that, for all archimedean places $v$ of $F$ or $v$ not in~$S$, we have $\varphi_v=\varphi_{0,v}$ and $f_v=f_{0,v}$. For every place $v$ of $F$, write $\mathcal{F}_v\coloneqq\theta(\varphi_v,f_v)$, and write $\mathcal{F}\coloneqq\otimes'_v\mathcal{F}_v$. For all archimedean places $v$ of $F$, Proposition \ref{prop:hwv} below indicates that $\mathcal{F}_v$ lies in the highest weight space of the minimal $K$-type of $\pi_v$. Fix the isomorphism $\pi_v\cong\pi_{n_v}$ such that $\mathcal{F}_v$ corresponds to $x_\ell^{2n_v}$.

\begin{theorem}\label{thm:general}
If $\mathcal{E}$ does not lie in $\X(\O_{F,S})$, then $a_{\mathcal{E}}(\mathcal{F}^\infty)=0$. After replacing $\mathcal{F}$ with a $\C^\times$-multiple, the following is true: for all $\mathcal{E}$ in $\X(F)$, if $\mathcal{E}$ lies in $\X(\O_{F,S})$ and corresponds to a cubic \'etale $F$-algebra $E$, then
\begin{itemize}
    \item if the cubic $\O_{F,S}$-algebra corresponding to $\mathcal{E}$ is $\O_{E,S}$ and $\epsilon_v=\epsilon_v(E_v,\chi_v,\psi_v)$ for all $v$ in $S$, then
    \begin{align*}
        |a_{\mathcal{E}}(\mathcal{F}^\infty)|^2 = L(\textstyle\frac12,\Ind^F_K\chi\otimes V_E)\cdot\displaystyle\prod_{v\mid\infty}q(\mathcal{E}_v)^{n_v-1/2}\cdot\prod_{\substack{v\nmid\infty\\ v\in S}}|\mathcal{I}_v(\mathcal{E}_v,\varphi_v,f_v)|^2\cdot \prod_{v\in S}C_v,
    \end{align*}
    where $V_E$ denotes the $2$-dimensional Artin representation associated with $E/F$, and $C_v$ is a nonzero constant depending only on $\chi_v$, $E_v$, and $\mu$,
    \item if $\epsilon_v\neq\epsilon_v(E_v,\chi_v,\psi_v)$ for some $v$ in $S$, then $a_{\mathcal{E}}(\mathcal{F}^\infty)=0$.
\end{itemize}
\end{theorem}
\begin{proof}
By our choice of isomorphism $\pi_v\cong\pi_{n_v}$ for all archimedean places $v$ of $F$, we have 
\begin{align}\label{eq:formula1}
\begin{split}\mathcal{F}_{N,\psi_{\mathcal{E}}}(1) &= a_{\mathcal{E}}(\mathcal{F}^\infty)\prod_{v\mid\infty}\mathcal{W}^{\mathcal{E}_v}_{n_v}(x_\ell^{2n_v})(1) \\
&= a_{\mathcal{E}}(\mathcal{F}^\infty)\prod_{v\mid\infty}\frac1{(2n_v)!}\left[\frac{|r_v(a_vi+b_v-c_vi-d_v)|}{r_v(a_vi+b_v-c_vi-d_v)}\right]^{-n_v}K_{-n_v}(|r_v(a_vi+b_v-c_vi-d_v)|),
\end{split}
\end{align}
where $(a_v,b_v,c_v,d_v)$ denotes $\mathcal{E}_v$.

If $\epsilon_v\neq\epsilon_v(E_v,\chi_v,\psi_v)$ for some $v$ in $S$, then \cite[Proposition 6.2]{ourselves} and the discussion from \S\ref{ss:Yang} show that $\mathcal{F}_{N,\psi_{\mathcal{E}}}(1)=0$ and hence $a_{\mathcal{E}}(\mathcal{F}^\infty)=0$. If $\epsilon_v=\epsilon_v(E_v,\chi_v,\psi_v)$ for all $v$ in $S$, then \cite[(6.1)]{ourselves} indicates that
    \begin{align}\label{eq:ourselves}
    \mathcal{F}_{N,\psi_\mathcal{E}}(1) = \theta(\varphi,f)_{N,\psi_{\mathcal{E}}}(1)=\int_{i(T_E)(\A_F)\backslash G'(\A_F)}\theta(g'\cdot\varphi)_{\widetilde{N},\psi_{\mathcal{X}}}(1)\overline{\mathcal{P}_i(g'\cdot f)}\dd{g'},
    \end{align}
where $i$ is any element of the $G'(F)$-orbit from \S\ref{ss:Yang}, $\mathcal{X}$ is the corresponding element of $\O_{\min}(F)\cap p^{-1}(\mathcal{E})$ under Lemma \ref{lem:integralminorbitfibers}, and $\mathcal{P}_i:\sigma\ra\one$ is the associated $i(T_E)(\A_F)$-invariant functional from \S\ref{ss:Yang}.

Recall from Proposition \ref{prop:Yang} the element $f_i$ of $\sigma$. First, we claim that, for all $f$ in $\sigma$ of the form $\otimes'_vf_v$,
    \begin{align}
    \mathcal{P}_i(g'\cdot f)=\mathcal{P}_i(f_i)\cdot\prod_v\beta_{\mathcal{X}_v}(g'_v\cdot f_v), \label{eqn:factoring_P_i}
    \end{align}
    where $v$ runs over places of $F$. To see this, note that $\mathcal{P}_i$ is an element of the $1$-dimensional space
    \begin{align*}
        \Hom_{i(T_E)(\A_F)}(\sigma,\one)=\sideset{}{'}\bigotimes_v\Hom_{i(T_E)(F_v)}(\sigma_v^{\epsilon_v},\one),
    \end{align*}
    so $\mathcal{P}_i$ is a $\C$-multiple of $\prod_v\beta_{\mathcal{X}_v}$. Evaluating at $f=f_i$ shows that this multiple is $\mathcal{P}_i(f_i)$, as desired.

    The claim~\eqref{eqn:factoring_P_i} and Lemma \ref{lemma:c_X(Omega)=1} show that (\ref{eq:ourselves}) equals
    \begin{align*}
    \mathcal{F}_{N,\psi_\mathcal{E}}(1)&=\overline{\mathcal{P}_i(f_i)}\cdot\prod_v\int_{i(T_E)(F_v)\backslash G'(F_v)}\alpha_{\mathcal{X}_v}(g'_v\cdot\varphi_v)\overline{\beta_{\mathcal{X}_v}(g'_v\cdot f_v)}\dd{g'_v} \\
    &=\overline{\mathcal{P}_i(f_i)}\cdot\prod_v\mathcal{I}_v(\mathcal{E}_v,\varphi_v,f_v),
    \end{align*}
    where $\mathcal{I}_v(\mathcal{E}_v,-,-)$ denotes the integral from Definition \ref{defn:localzetaintegral}. If $\mathcal{E}$ does not lie in $\X(\O_{F,S})$, then Proposition \ref{prop:unram_and_arch_int}.(1) shows that this vanishes. If $\mathcal{E}$ does lie in $\X(\O_{F,S})$ and corresponds to $\O_{E,S}$ for some \'etale $F$-algebra $E$, then Proposition \ref{prop:Yang} and Proposition \ref{prop:unram_and_arch_int} yield
    \begin{align}\label{eq:formula2}
    \begin{split}
    |\mathcal{F}_{N,\psi_{\mathcal{E}}}(1)|^2 &= |\mathcal{P}_i(f_i)|^2\cdot\prod_v|\mathcal{I}_v(\mathcal{E}_v,\varphi_v,f_v)|^2\\
&=L(\textstyle\frac12,\chi)\cdot L(\textstyle\frac12,\Ind^F_K\chi\otimes V_E)\cdot\Delta_{\O_E/\Z}^{1/2}\cdot\displaystyle\prod_{\substack{v\nmid\infty\\ v\in S}}|\mathcal{I}_v(\mathcal{E}_v,\varphi_v,f_v)|^2\cdot\prod_{v\in S}C_v\\
&\quad\cdot\abs{\prod_{v\mid\infty}q(\mathcal{E}_v)^{(n_v-1)/2}\left[\frac{|r_v(a_vi+b_v-c_vi-d_v)|}{r_v(a_vi+b_v-c_vi-d_v)}\right]^{-n_v}K_{-n_v}(|r_v(a_vi+b_v-c_vi-d_v)|)}^2.
\end{split}
\end{align}
Finally, by rewriting
\begin{align*}
\Delta_{\O_E/\Z} = \Delta_{\O_F/\Z}^3\cdot\Nm_{\O_F/\Z}(\Delta_{\O_E/\O_F}) = \Delta_{\O_F/\Z}^3\cdot\prod_{v\mid\infty}q(\mathcal{E}_v),
\end{align*}
the desired result follows from comparing (\ref{eq:formula2}) with (\ref{eq:formula1}).
\end{proof}

For nonarchimedean $v$ in $S$, we can explicitly calculate $\mathcal{I}_v(\mathcal{E}_v,\varphi_v,f_v)$ for certain $\varphi_v$ and $f_v$, which yields:

\begin{theorem}\label{thm:refined}
We can choose $\varphi$ and $f$ such that the following is true: for all $\mathcal{E}$ in $\X(F)$,
\begin{itemize}
\item if $\mathcal{E}$ does not lie in $\X(\O_F)$, then $a_{\mathcal{E}}(\mathcal{F}^\infty)=0$,
\item if $\mathcal{E}$ lies in $\X(\O_F)$ and corresponds to the ring of integers of a cubic \'etale $F$-algebra $E$, then
\begin{align*}
|a_{\mathcal{E}}(\mathcal{F}^\infty)|^2 = \begin{cases}
L(\textstyle\frac12,\Ind^F_K\chi\otimes V_E)\cdot\displaystyle\prod_{v\mid\infty}q(\mathcal{E}_v)^{n_v-1/2} & \text{if $\epsilon_v=\epsilon(E_v,\chi_v,\psi_v)$ for all $v$ in $S$},\\
0 & \mbox{otherwise.}
\end{cases}
\end{align*}
\end{itemize}
\end{theorem}
\begin{proof}
For all $v$ in $S$ and cubic \'etale $F_v$-algebras $E_v$, write $C_{E_v, v}$ for the resulting nonzero constants from Theorem \ref{thm:general} (which depend on $\mu$). For archimedean $v$ in $S$, replace $\varphi_{0,v}$ or $f_{0,v}$ with its $C_{\RR^3,v}^{-1}$-multiple. For nonarchimedean $v$ in $S$, apply 
 Proposition \ref{prop:ramified_local_integrals} below with these constants $C_{E_v,v}$ to obtain elements $\varphi_v$ of $\Omega_v$ and $f_v$ of $\sigma_v^{\epsilon_v}$. Finally, the result follows from applying Theorem \ref{thm:general} to our resulting $\varphi$ and $f$.
\end{proof}
\begin{remark}
If there exists a positive integer $n$ such that $n=n_v$ for all archimedean places $v$ of $F$, and if $\mathcal E$ lies in $\mathbb X(\O_F)$ with associated cubic $\O_F$-algebra $\O_{\mathcal E}$, then $$\prod_{v\mid\infty} q(\mathcal E_v)^{n_v - 1/2} = \Nm_{\O_F/\Z}(\Delta_{\mathcal O_{\mathcal E}/\O_F})^{n-1/2}.$$
\end{remark}

\section{Unramified test vectors and local integrals}\label{s:unramified}

In this section, assume that $F$ is a nonarchimedean local field and that $K/F$ is unramified. Then $\widetilde{G}$ and its various subvarieties that we consider have natural models over $\O_F$. For instance, recall $J(\O_F)$ from Example~\ref{exmp}. Also, assume that $\chi$ is unramified and that $\psi$ has conductor $0$.

Our goal is to prove Theorem \ref{prop:unramified_integral}, which shows that our local integral evaluates to 1 in the unramified, locally maximal setting. We begin in \S\ref{ss:sphericalvectorE6} by defining and studying our spherical vector $\varphi_0$ in $\Omega$. We calculate local Fourier coefficients of $\varphi_0$ in terms of certain subalgebras of $J(\O_F)$, and we analyze how these subalgebras behave in \S\ref{ss:conjugates}. Finally, in \S\ref{ss:unramifiedlocalintegral} we use these results to compute our local integral.

\subsection{Spherical vectors for $\widetilde{G}$}\label{ss:sphericalvectorE6}
For all $\mathcal X$ in $\O_{\min}(F)$, write $v(\mathcal X)$ for the unique integer $n$ such that $\mathcal X$ lies in $\varpi^n\widetilde\X(\O_F) - \varpi^{n+1}\widetilde\X(\O_F)$. 
\begin{lemma}\label{lem:integral_O_min_orbits}
    The sets $S_n \coloneqq \left\{\mathcal X\in \O_{\min}(F) \mid v(\mathcal X) = n\right\}$ are precisely the $\widetilde M(\O_F)$-orbits in $\O_{\min}(F)$. 
\end{lemma}
\begin{proof}
    We follow the discussion in \cite[\S3.4]{kazhdan2004minimal}, with some simplifications. Since $\widetilde{M}$ acts linearly on $\widetilde\X$, we see that $\widetilde M(\O_F)$ preserves $S_n$. To show that $S_n$ is a single orbit, write $Q \subseteq \widetilde M$ for the stabilizer of the line spanned by $(1,0,0,0)$ in $\widetilde\X(F)$. Write $S\subseteq Q$ for the stabilizer of $(1,0,0,0)$, write $\lambda:\G_m\hookrightarrow\widetilde{M}$ for the cocharacter given by scaling on $\widetilde{\X}$, and note that $Q=S\times\lambda(\G_m)$. Because $Q$ is a parabolic subgroup of $\widetilde{M}$, the Iwasawa decomposition yields $\widetilde{M}(F)=\widetilde{M}(\O_F)S(F)\lambda(F^\times)$.

Let $\mathcal{X}$ be in $S_n$. Then Lemma \ref{lem:transitive_O_min} and the above decomposition imply that there exists $t$ in $F^\times$ such that~$\mathcal{X}$ lies in the $\widetilde{M}(\O_F)$-orbit of $\lambda(t)\cdot(1,0,0,0)=(t,0,0,0)$. Since $\widetilde{M}(\O_F)$ stabilizes $S_n$, this indicates that $v(t)=n$. Therefore $\mathcal{X}$ lies in the $\widetilde{M}(\O_F)$-orbit of $(\varpi^n,0,0,0)$, i.e. $S_n$ is indeed a single orbit.
\end{proof}

\begin{lemma}\label{lem:spherical_coinvariants_for_normalization}
    For any nonzero $\widetilde{G}(\O_F)$-fixed element $\varphi$ of $\Omega$, the image of $\varphi$ in $\Omega_{\widetilde N(F), \psi_{(1, 0, 0,0)}}$ is nonzero.
\end{lemma}
\begin{proof}
We will use the model for $\Omega$ from \cite{kazhdan2004minimal} in the split case (for the non-split case, see \cite{rumelhart1997minimal}). In this model, there is a certain continuous representation of $\widetilde{G}(F)$ on $L^2(F^\times\times F\times J)$, and $\Omega$ is its subspace of smooth vectors. Moreover, using a pinning of $\widetilde{G}$ over $\O_F$, we get an isomorphism $z:\G_a\ra^\sim Z$ of groups over $\O_F$ and a section $\widetilde{n}:\widetilde{N}/Z\hookrightarrow\widetilde{N}$ over $\O_F$ of the quotient morphism such that, for all $f$ in $\Omega$, we have
\begin{align}
        (z(t) f)(y, x_0, x) &= \psi(ty) \label{eq:first_Omega_formula}f(y, x_0, x),\\
        (\widetilde{n}(a, b, 0, 0)f)(y, x_0, x) &= f(y, x_0 + ay, x + by), \label{eq:second_Omega_formula}\\
        (\widetilde{n}(0, 0, c, d)f)(y, x_0, x) &= \psi(-\tr(c\circ x) + dx_0) f(y, x_0, x)\label{eq:third_Omega_formula}
    \end{align}
\cite[p.~5]{kazhdan2004minimal}.\footnote{The formulas in \cite[Proposition 41]{rumelhart1997minimal} look slightly different and in particular omit the minus sign in (\ref{eq:third_Omega_formula}). This arises from us using a different identification between $\widetilde N /Z$ and $F \oplus J \oplus J \oplus F$ than the one in \cite{rumelhart1997minimal}; compare \cite[Proposition 10]{rumelhart1997minimal} with (\ref{eq:our_symplectic_form}).}
Also, for a certain pair of Weyl elements $A$ and $S$ in $\widetilde G(\O_F)$, we have
\begin{align}\label{eq:fourth_Omega_formula}
    (Sf)(y, x_0, x) &= \int_{F\times J} f(y, x_0', x') \psi\Big(\frac{x_0' x_0 + \tr(x'\circ x)}{y}\Big) |y|^{-5} \dd{x_0'} \dd{x'}, \\
    \label{eq:fifth_Omega_formula}
    (Af)(y, x_0, x) &= \psi(-\det(x)/(x_0 y)) f(-x_0, y, x),
\end{align}
where (\ref{eq:fifth_Omega_formula}) only holds for $x_0\neq0$ \cite[p.~5]{kazhdan2004minimal}.

For all $\mathcal X = (a, x, x^\#/ a, \det (x)/a^2)$ in $\O_{\min}(F)$ with $a \neq 0$, one can use (\ref{eq:third_Omega_formula}) to show that the functional
\begin{equation}
\label{eq:sixth_Omega_formula}
    f\mapsto\lim_{x_0\to 0} \psi(-\det(x)/(x_0a)) f(-x_0, a, x) =  (A f)(a, 0, x)
\end{equation}
yields a nonzero element of $\Hom_{\widetilde{N}(F)}(\Omega,\psi_{\mathcal{X}})$. In particular, $\alpha_{\mathcal{X}}$ is a $\C^\times$-multiple of (\ref{eq:sixth_Omega_formula}).

Let $\varphi$ be a $\widetilde{G}(\O_F)$-fixed element of $\Omega$. Because $\psi$ has conductor $0$, (\ref{eq:first_Omega_formula}) and (\ref{eq:third_Omega_formula}) imply that $\varphi$ is supported on $(\O_F-\{0\})\times\O_F\times J(\O_F)$. If $\alpha_{(1,0,0,0)} (\varphi) = 0$, then we will show by induction on non-negative integers $n$ that $\varphi(y,x_0,x)=0$ whenever $v(y)=n$. As $n$ varies, this will imply that $\varphi=0$, which completes the proof.

First, consider the base case $n=0$. 
Because $v(1,0,0,0)=0$ and $\alpha_{(1,0,0,0)}(\varphi)=0$, Lemma \ref{lem:integral_O_min_orbits} and Lemma \ref{lem:alpha_equivariance} indicate that $\alpha_{\mathcal{X}}(\varphi)=0$ for all $\mathcal{X}$ with $v(\mathcal{X})=0$. For all $y$ in $\O^\times_F$, (\ref{eq:second_Omega_formula}) implies that $(x_0,x)\mapsto\varphi(y,x_0,x)$ is invariant under translation by $\O_F\times J(\O_F)$, so it suffices to check that $\varphi(y,0,0)=0$. Since $A\varphi=\varphi$, (\ref{eq:sixth_Omega_formula}) shows that this is equivalent to $\alpha_{(y,0,0,0)}(\varphi)=0$, and the latter follows from $v(y,0,0,0)=0$.

Next, consider the inductive step. Let $y$ be in $\varpi^{n+1}\O_F^\times$, and by induction assume that $\varphi(y',x_0,x)=0$ whenever $v(y')\leq n$. Since $A\varphi=\varphi$, (\ref{eq:sixth_Omega_formula}) shows that $\alpha_{(y',0,0,0)}(\varphi)=0$. Hence Lemma \ref{lem:integral_O_min_orbits} and Lemma \ref{lem:alpha_equivariance} indicate that $\alpha_{\mathcal{X}}(\varphi)=0$ for all $\mathcal{X}$ with $v(\mathcal{X})\leq n$. Applying $A\varphi=\varphi$ and (\ref{eq:sixth_Omega_formula}) again yields $\varphi(y,0,x)=0$ for all $x$ in $J(F)-\varpi^{n+1}J(\O_F)$. Applying $A\varphi=\varphi$, (\ref{eq:fifth_Omega_formula}), and the induction hypothesis indicates that $\varphi(y,x_0,x)=0$ whenever $v(x_0)\leq n$, so altogether $(x_0,x)\mapsto\varphi(y,x_0,x)$ is supported on $(\varpi^{n+1}\O_F)\times(\varpi^{n+1}J(\O_F))$. Therefore (\ref{eq:fourth_Omega_formula}) implies that $(x_0,x)\mapsto(S\varphi)(y,x_0,x)$ is invariant under translation by $\O_F\times J(\O_F)$. Finally, $S\varphi=\varphi$, so we see that $\varphi(y,x_0,x)=0$ for all $(x_0,x)$ in $\O_F\times J(\O_F)$, as desired.
\end{proof}

Write $\varphi_0$ for the unique $\widetilde{G}(\O_F)$-fixed element of $\Omega$ satisfying $\alpha_{(1,0,0,0)}(\varphi_0)=1$, which exists by Lemma \ref{lem:spherical_coinvariants_for_normalization}.

\begin{corollary}\label{cor:dumbcorsphericalvanishing}
      For all $\mathcal X$ in $\O_{\min}(F)$, we have $\alpha_{\mathcal X}(\varphi_0) = 0$ if $v(\mathcal X) < 0$ and $\alpha_{\mathcal X}(\varphi_0) = 1$ if $v(\mathcal X) = 0$. 
\end{corollary}
\begin{proof}
    If $v(\mathcal X) < 0$, then $\psi_{\mathcal X} :\widetilde{N}(F)\to \C^1$ is nontrivial on $\widetilde{N}(\O_F)$ because $\psi$ has conductor $0$. In particular, the $\widetilde{N}(\O_F)$-invariance of $\varphi_0$ implies that $\alpha_{\mathcal X}(\varphi_0) = 0$.

    If $v(\mathcal X) = 0$, the claim follows from the $\widetilde G(\O_F)$-invariance of $\varphi_0$ along with Lemma \ref{lem:alpha_equivariance} and Lemma \ref{lem:integral_O_min_orbits}.
\end{proof}

\begin{corollary}\label{cor:calculate_spherical_easy}
Let $\mathcal{E}=(a,b,c,d)$ be an element of $\X(\O_F)$ such that the associated cubic $\O_F$-algebra $\O_{\mathcal{E}}$ is the ring of integers of a cubic \'etale $F$-algebra. Then for all $\mathcal{X}=(a,x,y,d)$ in $\O_{\min}(F)\cap p^{-1}(\mathcal{E})$, we have 
$$\alpha_{\mathcal X} (\varphi_0) = \begin{cases}
        1  & \mbox{when }x\mbox{ and }y\mbox{ lie in }J(\O_F), \\
        0 & \text{otherwise}.
    \end{cases}$$
\end{corollary}
\begin{proof}
If $x$ and $y$ do not both lie in $J(\O_F)$, then 
$v(\mathcal X) < 0$, so the claim follows from Corollary \ref{cor:dumbcorsphericalvanishing}. 

Now assume that $x$ and $y$ lie in $J(\O_F)$. By Corollary \ref{cor:dumbcorsphericalvanishing}, it suffices to show that $v(\mathcal{X})=0$, i.e. that $(a,x,y,d)\not\equiv 0\pmod{\varpi}$. If we had $(a,x,y,d)\equiv0\pmod{\varpi}$, then $(a,b,c,d)\equiv0\pmod{\varpi}$, so that $\O_{\mathcal{E}}/\varpi$ is isomorphic to $(\O_F/\varpi)[\alpha,\beta]/(\alpha^2,\beta^2,\alpha\beta)$ by \cite[p.~115]{GGS02}. In particular, the $\O_F$-algebra $\O_{\mathcal{E}}$ is ramified and not monogenic. But $\O_{\mathcal{E}}$ is the ring of integers of a cubic \'etale $F$-algebra, so $\O_{\mathcal{E}}$ is only not monogenic when $\O_F/\varpi=\mathbb{F}_2$ and $\O_{\mathcal{E}}$ is split over $\O_F$. But this $\O_{\mathcal{E}}$ is unramified, so altogether we must have $v(\mathcal{X})=0$.
\end{proof}

\subsection{Conjugates of algebra embeddings}\label{ss:conjugates}
Let $i:E\hookrightarrow J$ be an $F$-algebra embedding such that $i(\O_E)$ lies in $J(\O_F)$. To apply Corollary \ref{cor:calculate_spherical_easy}, we will need a criterion for when conjugates of $i(\O_E)$ remain in $J(\O_F)$. First, we tackle the case when $K=F\times F$.
\begin{lemma}\label{lem:double_coset}
If $g'$ in $\GL_3(F)$ satisfies $g'^{-1}i(\O_E) g'\subseteq\operatorname{M}_3(\O_F)$, then $g'$ lies in $i(E^\times)\GL_3(\O_F)$. 
\end{lemma}
Our original proof of Lemma \ref{lem:double_coset} used tedious casework on the ramification behavior of $E/F$. We are very grateful to Aaron Pollack for explaining to us the following, much simpler proof.
\begin{proof}
The two $\O_F$-algebra embeddings $\O_E\hookrightarrow\operatorname{M}_3(\O_F)$ given by $i$ and $\ad{g'^{-1}}\circ i$ both endow $\O_F^{3}$ with an $\O_E$-module structure extending its $\O_F$-module structure. In both $\O_E$-module structures, $\O_F^{3}$ is evidently finitely generated and torsionfree, so it is finite free over $\O_E$. These two $\O_E$-module structures have the same rank over $\O_E$ and hence are isomorphic, so there exists an $\O_F$-linear isomorphism $k:\O_F^{3}\ra^\sim\O_F^{3}$ sending the action of $i$ to the action of $\ad{g'^{-1}}\circ i$. This implies that $g'k$ centralizes $i(E^\times)$. But $i(E^\times)$ is its own centralizer in $\GL_3(F)$, which yields the desired result.
\end{proof}
Next, we use Lemma \ref{lem:double_coset} to deduce the case when $K$ is a field.
\begin{prop}\label{prop:doublecoset}
If $g'$ in $\U_3(F)$ satisfies $g'^{-1}i(\O_E)g'\subseteq J(\O_F)$, then $g'$ lies in $i(L^1)\U_3(\O_F)$.
\end{prop}
\begin{proof}
    When $K = F\times F$, this is Lemma \ref{lem:double_coset}, so assume that $K$ is a field. If $g'^{-1}i(\O_E)g'\subseteq J(\O_F)$, then applying Lemma \ref{lem:double_coset} to $L/K$ shows that $g'=hk$ for some $h$ in $i(L^\times)$ and $k$ in $\GL_3(\O_K)$. Since $i(L^1)\subseteq\U_3(F)$ and $i(\O_L^\times)\subseteq\GL_3(\O_K)$, it suffices to prove that $h$ lies in $i(L^1\O_L^\times)$.

    We have $g'=\prescript{t}{}{\overline{g'}}^{-1}$, so $\prescript{t}{}{\overline{h}}h=\prescript{t}{}{\overline{k}}^{-1}k^{-1}$ lies in $i(L^\times)\cap\GL_3(\O_K)=i(\O_L^\times)$. Because $i:L\hookrightarrow\operatorname{M}_3(K)$ is an embedding of $K$-algebras with involution, it suffices to prove that, for all $h_0$ in $L^\times$, if $\overline{h_0}h_0$ lies in $\O_L^\times$, then $h_0$ lies in $L^1\O_L^\times$. By writing $E$ as a product of fields, we can assume that $E$ is a field. Finally,
    \begin{itemize}
        \item If $L$ is a field, then $L/E$ is an unramified quadratic extension. Therefore if $\overline{h_0}h_0$ lies in $\O_L^\times$, then $h_0$ lies in $\O_L^\times$.
        \item If $L=E\times E$, then $h_0=(h_1,h_2)$ for some $h_1$ and $h_2$ in $E^\times$. Hence if $\overline{h}_0h_0=(h_1h_2,h_1h_2)$ lies in $\O_L^\times=(\O_E^\times)^2$, then $h_0=(h_1,h_1^{-1})(1,h_1h_2)$ lies in $L^1\O_L^\times$.\qedhere
    \end{itemize}
    \end{proof}

\subsection{Unramified local integrals}\label{ss:unramifiedlocalintegral}
With Proposition \ref{prop:doublecoset} in hand, we are ready to finish calculating the unramified local integrals. Recall from \S\ref{ss:localPU3} that, under the unramified hypotheses of this section, when $K$ is a field we have $\chi^2=1$, so the sign $\epsilon$ equals $+1$. Recall from \S\ref{ss:G'localvector} the element $\lambda_0$ of $E^\times/\Nm_{L/E}(L^\times)$.

\begin{lemma}\label{lem:choose_X_integrally}
There exists an isomorphism $K^3 \cong L_{\lambda_0}$ of Hermitian spaces for $K/F$ such that
 \begin{itemize}
     \item the image of $\O_K^3$ in $L_{\lambda_0}$ is $\O_L$-stable,
     \item the associated $F$-algebra embedding $i:E\hookrightarrow J$ satisfies $\Hom_{i(T_E)(F)} (\sigma^+, \mathbbm 1) \neq 0$.
 \end{itemize}
 In particular, $i(\O_E)$ lies in $J(\O_F)$.
 \end{lemma}
\begin{proof}
By Propositon \ref{prop:epsilon_and_i}, the $i$ satisfying $\Hom_{i(T_E)(F)}(\sigma^+,\one)\neq0$ are precisely those arising from $\lambda_0$ and an isomorphism $K^3\cong L_{\lambda_0}$ of Hermitian spaces for $K/F$ as in \S\ref{ss:unitarysetup}, with  $\lambda_0$ as in \S\ref{ss:G'localvector}. If there exists an $\O_L$-lattice $M$ in $L_{\lambda_0}$ which is self-dual with respect to the Hermitian form on $K^3$, then we can choose the isomorphism $K^3\cong L_{\lambda_0}$ to send $\O_K^3$ to $M$. Hence it suffices to prove that such an $M$ exists.

 When $E = F\times F\times F$, our $\lambda_0$ is represented by $1$, so we can take $M = \O_K^3 \subseteq L_{\lambda_0} = K^3$.

 When $E/F$ is a field, note that the inclusion and norm maps induce mutually inverse isomorphisms between $E^\times / \Nm_{L/E}(L^\times)$ and $F^\times / \Nm_{K/F} (K^\times)$. Write $d(E/F)$ for the valuation of the different of $\O_E/\O_F$ (equivalently, of $\O_L/\O_K$). Because the norm of the different is the discriminant, $\lambda_0$ is represented by $\varpi_E^{-d(E/F)}$.  Using this, we see that we can take $M = \O_L$. 

When $E=F\times F'$ for a field $F'$, write $K'\coloneqq K\otimes_FF'$, and write $d(K'/K)$ for the valuation of the different of $\O_{K'}/\O_K$. Then we see that $\lambda_0$ is represented by $(1,\varpi_{F'}^{d(K'/K)})$ in 
\begin{align*}
E^\times/\Nm_{L/E}(L^\times)=(F^\times/\Nm_{K/F}(K^\times))\times(F'^\times/\Nm_{K'/F'}(K'^\times))\subset \{\pm1\}\times\{\pm1\}.
\end{align*}
Hence $M=\O_K\times\varpi_{K'}^{-d(K'/K)}\O_{K'}$ yields the desired $\O_L$-lattice.
\end{proof}
Recall from \S\ref{ss:G'localvector} that, under the unramified hypotheses of this section, the sign $\epsilon(E,\chi,\psi)=+1$.

Recall from \S\ref{ss:G2local} that $\mathcal{E}$ denotes an element of $\X(F)$ such that the associated cubic $F$-algebra is isomorphic to $E$. When $\mathcal{E}$ lies in $\X(\O_F)$, write $\O_{\mathcal{E}}$ for the associated cubic $\O_F$-algebra. Recall from \S\ref{ss:localPU3} the nonzero $G'(\O_F)$-fixed element $f_0$ of $\sigma^+$, and recall from Definition \ref{defn:localzetaintegral} the integral $\mathcal{I}(\mathcal{E},-,-)$.
\begin{thm}\label{prop:unramified_integral}
We have
\begin{align*}
|\mathcal{I}(\mathcal{E},\varphi_0,f_0)|=\begin{cases}
1 & \mbox{if }\mathcal{E}\mbox{ lies in }\X(\O_F)\mbox{ and }\O_{\mathcal{E}}\cong\O_E,\\
0 & \mbox{if }\mathcal{E}\mbox{ does not lie in }\X(\O_F).
\end{cases}
\end{align*}
\end{thm}
\begin{proof}
Since $G'$ lies in $\widetilde{M}^1$, Lemma \ref{lem:alpha_equivariance} yields
\begin{align*}
\mathcal I(\mathcal E, \varphi_0, f_0) &=\int_{i(T_E)(F)\backslash G'(F)} \alpha_{\mathcal X} (g'\cdot\varphi_0)\overline{\beta_{\mathcal{X}}(g'\cdot f_0)}\dd{g'} \\
&= \int_{i(T_E)(F)\backslash G'(F)} \alpha_{g'^{-1}\cdot\mathcal X} (\varphi_0) 
\overline{\beta_{\mathcal X}(g'\cdot f_0)}\dd{g'}.
\end{align*}
If $\mathcal E$ does not lie in $\mathbb X(\O_F)$, then for any $\mathcal X$ in $\O_{\min}(F)\cap p^{-1} (\mathcal E)$, we have $v(g'^{-1} \cdot \mathcal X) < 0$ for all $g'$ in $G'(F)$ because $p$ sends $\widetilde{\X}(\O_F)$ to $\X(\O_F)$. Together with Corollary \ref{cor:dumbcorsphericalvanishing}, this proves the second case.

For the first case, choose $i$ satisfying the conclusion of Lemma \ref{lem:choose_X_integrally}, and write $\mathcal{X}=(a,x,y,d)$ for the corresponding element of $\O_{\min}(F)\cap p^{-1}(\mathcal{E})$ under Lemma \ref{lem:integralminorbitfibers}. In particular, $x$ and $y$ lie in $J(\O_F)$. Then Corollary \ref{cor:calculate_spherical_easy} and Proposition \ref{prop:doublecoset} show that
\begin{align*}
\alpha_{g'^{-1}\cdot \mathcal X} (\varphi_0) = \begin{cases}
1 & \mbox{when }g'\mbox{ lies in }i(T_E)(F)G'(\O_F),\\
0 & \mbox{otherwise.}
\end{cases}
\end{align*}
Because $G'(\O_F)$ fixes $f_0$, Definition \ref{defn:beta} indicates that $\beta_{\mathcal X} (g' \cdot f_0) = \beta_{\mathcal X} (f_0)$ for all $g'$ in $i(T_E)(F) G'(\O_F)$. Therefore Lemma \ref{lem:spherical_compare_Yang} implies that the absolute value of our integral equals
$$
|\mathcal I(\mathcal E, \varphi_0, f_0)| = \vol\left(i(T_E)(F) \backslash i(T_E)(F)G'(\O_F)\right) = \vol\left(i(T_E)(F) \cap G'(\O_F) \backslash G'(\O_F)\right).$$
Since $i(\O_E)$ lies in $J(\O_F)$, we see that $i(T_E)(F)\cap G'(\O_F)$ is the maximal compact subgroup of $i(T_E)(F)$. Hence our choice of measures yields $\vol(i(T_E)(F)\cap G'(\O_F)\backslash G'(\O_F))=1$, as desired.
\end{proof}

\section{Ramified test vectors and local integrals}\label{s:ramified}
In this section, assume that $F$ is a nonarchimedean local field. Our goal is to prove Proposition \ref{prop:ramified_local_integrals}, which lets us choose local vectors with particularly nice local integrals.

Recall from \S\ref{ss:G2local} that $\mathcal{E}$ denotes an element of $\X(F)$ such that the associated cubic $F$-algebra $E$ is \'etale. When $\mathcal{E}$ lies in $\X(\O_F)$, write $\O_{\mathcal{E}}$ for the associated cubic $\O_F$-algebra.
\begin{lemma}\label{lem:cts_section}
Assume that $\mathcal{E}$ lies in $\X(\O_F)$, and write $\mathcal{M}$ for the stabilizer of $\mathcal{E}$ in $M(\O_F)$. There exists a continuous section $s:M(\O_F)/\mathcal{M}\ra M(\O_F)$ of the quotient whose image is a compact neighborhood of $1$.\end{lemma}
\begin{proof}
The discussion from \S\ref{ss:G2local} shows that $\mathcal{M}$ is isomorphic to $\Aut_{\O_F}(\O_{\mathcal{E}})$ and hence is finite. Therefore the quotient map $M(\O_F)\ra M(\O_F)/\mathcal{M}$ is finite \'etale. Because $M(\O_F)/\mathcal{M}$ is profinite, this map has a continuous section $s:M(\O_F)/\mathcal{M}\ra M(\O_F)$, which is \'etale and hence an open embedding. Finally, after replacing $s$ with an $\mathcal{M}$-translate, we can assume that its image indeed contains $1$.
\end{proof}

Recall from \S\ref{ss:seesaw} the sign $\epsilon$ and the irreducible smooth representation $\sigma^{\epsilon}$ of $G'(F)$,
recall from \S\ref{ss:G'localvector} the sign $\epsilon(E,\chi,\psi)$, and assume that $\epsilon=\epsilon(E,\chi,\psi)$. Then Proposition \ref{prop:epsilon_and_i} shows there exists a unique
$G'(F)$-orbit of~$i$ in $\{E\hookrightarrow J\}$ satisfying $\Hom_{i(T_E)(F)}(\sigma^{\epsilon},\one)\neq0$. Let $i$ be in this $G'(F)$-orbit, and recall the element $\beta_i$ in $\Hom_{i(T_E)(F)}(\sigma^\epsilon, \one)$ from Definition \ref{defn:beta}.

Recall from Definition \ref{defn:localzetaintegral} the integral $\mathcal{I}(\mathcal{E},-,-)$.

\begin{lemma}\label{lem:ramified_local_integrals}
Let $\O_0$ be a cubic $\O_F$-subalgebra of $\O_E$, and let $f$ in $\sigma^\epsilon$ be an element such that $\beta_i(f) \neq 0$. Then there exists $\varphi_{\O_0}$ in $\Omega$ such that
\begin{align*}
\mathcal{I}(\mathcal{E},\varphi_{\O_0},f) = \begin{cases}
1 & \mbox{if }\mathcal E\mbox{ lies in }\mathbb X(\O_F)\mbox{ and }\O_{\mathcal{E}}\cong\O_0, \\
0 & \mbox{otherwise.}
\end{cases}
\end{align*}
\end{lemma}
\begin{proof}
Let $\mathcal{E}_0$ be an element of $\X(\O_F)$ such that the associated cubic $\O_F$-algebra is isomorphic to $\O_0$. Write~$\mathcal{X}_0$ for the element of $\O_{\min}(F)\cap p^{-1}(\mathcal{E}_0)$ corresponding to $i$ under Lemma \ref{lem:integralminorbitfibers}, and let $K'$ be a compact open subgroup of $G'(F)$ that fixes $f$. Write $\mathcal{M}_0$ for the stabilizer of $\mathcal{E}_0$ in $M(\O_F)$, write $s_0:M(\O_F)/\mathcal{M}_0\ra M(\O_F)$ for the section of $M(\O_F)\ra M(\O_F)/\mathcal{M}_0$ from Lemma \ref{lem:cts_section}, and write $U$ for the image of $s_0$.

We claim that $(U\times K')\cdot\mathcal{X}_0$ is a compact open subset of $\O_{\min}(F)$. To see this, it suffices to prove that the map $M\times G'\!/i(T_E)\ra\O_{\min}$ over $F$ given by $(m,g')\mapsto mg'\cdot\mathcal{X}_0$ is an open embedding in a neighborhood of~$1$. By dimension counting, it suffices to show it is injective in a neighborhood of $1$. If $mg'\cdot\mathcal{X}_0=\mathcal{X}_0$, then
\begin{align*}
m\cdot\mathcal{E}_0 = m\cdot p(\mathcal{X}_0) = m\cdot p(g'\cdot\mathcal{X}_0) = p(mg'\cdot\mathcal{X}_0) = p(\mathcal{X}_0) = \mathcal{E}_0.
\end{align*}
The discussion from \S\ref{ss:G2local} indicates that the stabilizer of $\mathcal{E}_0$ in $M$ is finite, so in a Zariski neighborhood of $1$ the above implies that $m=1$. Then $g'\cdot\mathcal{X}_0=\mathcal{X}_0$, so $g'$ lies in $i(T_E)$, concluding the proof of the claim.

Let $\varphi_{\O_0}$ be an element of $\Omega$ whose image in $\Omega_{Z(F)}$ corresponds to the indicator function of $(U\times K')\cdot\mathcal{X}_0$ under the injection $C^\infty_c(\O_{\min}(F))\hookrightarrow\Omega_{Z(F)}$ from \S\ref{sss:nonarchE6}.

Now choose an element $\mathcal X\in \O_{\min} (F) \cap p^{-1} (\mathcal E)$, with which we will calculate the local integral.
If $\mathcal{E}$ does not lie in $\X(\O_F)$ or $\O_{\mathcal{E}}$ is not isomorphic to $\O_0$, then $\mathcal{E}$ is not in the $M(\O_F)$-orbit of $\mathcal{E}_0$. This implies that $\alpha_{\mathcal{X}}(g'\cdot\varphi_{\O_0})=0$ for all $g'$ in $G'(F)$, so $\mathcal{I}(\mathcal{E},\varphi_{\O_0},f)=0$.

 If $\mathcal{E}$ lies in $\X(\O_F)$ and $\O_{\mathcal{E}}$ is isomorphic to $\O_0$, then the discussion from \S\ref{ss:G2local} indicates that there is a unique $u$ in $U$ such that $u\cdot\mathcal{E}_0=\mathcal{E}$. In particular, we can take $\mathcal X = u\cdot \mathcal X_0$ for the definition of the local integral, so
\begin{align*}
\mathcal{I}(\mathcal{E},\varphi_{\O_0},f) = \int_{i(T_E)(F)\backslash G'(F)}\alpha_{u\cdot\mathcal{X}_0}(g'\cdot\varphi_{\O_0})\overline{\beta_{u\cdot\mathcal{X}_0}(g'\cdot f)}\dd{g'}.
\end{align*}
Now Definition \ref{defn:alphanonarch} shows that
\begin{align*}
\alpha_{u\cdot\mathcal{X}_0}(g'\cdot\varphi_{\O_0}) = \begin{cases}
1 & \mbox{when }ug'^{-1}\cdot\mathcal{X}_0\mbox{ lies in }(U\times K')\cdot\mathcal{X}_0,\\
0 & \mbox{otherwise.}
\end{cases}
\end{align*}
We claim that $ug'^{-1}\cdot\mathcal{X}_0$ lies in $(U\times K')\cdot\mathcal{X}_0$ if and only if $g'$ lies in $i(T_E)(F)K'$. To see this, if $ug'^{-1}\cdot\mathcal{X}_0=u_1k_1\cdot\mathcal{X}_0$ for some $u_1$ in $U$ and $k_1$ in $K'$, then $u\cdot\mathcal{E}_0=u_1\cdot\mathcal{E}_0$, so $u=u_1$ by uniqueness. This implies that $g'^{-1}\cdot\mathcal{X}_0=k_1\cdot\mathcal{X}_0$. Then $g'k_1$ stabilizes $\mathcal{X}_0$, so it lies in $i(T_E)(F)$, concluding the proof of the claim.

The claim indicates that our integral equals
\begin{align*}
\int_{i(T_E)(F)\backslash i(T_E)(F)K'}\overline{\beta_{u\cdot\mathcal{X}_0}(g'\cdot f)}\dd{g'}.
\end{align*}
Because $K'$ fixes $f$, we obtain $$\mathcal I(\mathcal E, \varphi_{\mathcal O_0}, f) = \overline{\beta_{u\cdot\mathcal X_0} (f)}\vol(i(T_E)(F) \backslash i(T_E)(F) K').$$ Finally, dividing $\varphi_{\O_0}$ by the above constant yields the desired result.
\end{proof}
Now we allow $E$ to vary. Suppose we are given, for each isomorphism class of cubic \'etale $F$-algebras $E$, an element $C_E$ of $\C^\times$.
\begin{prop}\label{prop:ramified_local_integrals}
    There exists $\varphi_0$ in $\Omega$ and     
    $f_0$ in $\sigma^\epsilon$ with the following property: for all $\mathcal E$ in $\X(F)$, we have
    $$\mathcal I(\mathcal E, \varphi_0, f_0) = \begin{cases}
        C_E^{-1} & \text{if $\mathcal E$ lies in $\X(\O_F)$ and $\O_{\mathcal E} \cong \O_E$ for some cubic \'etale $F$-algebra $E$ with $\epsilon(E, \chi, \psi) = \epsilon$,}\\
        0 & \text{otherwise.}
    \end{cases}$$
\end{prop}
\begin{proof}
There are finitely many (isomorphism classes of) cubic \'etale $F$-algebras $E$, so in particular finitely many $E$ satisfy $\epsilon=\epsilon(E,\chi,\psi)$. For each such $E$, fix a corresponding $i: E \hookrightarrow J$ and hence $\beta_i$ as above. Because each $\beta_i$ is nontrivial and $\C$ is infinite, there exists an $f_0$ in $\sigma^\epsilon$ such that $\beta_i(f_0) \neq 0$ for each such $E$. Finally, for each such $E$ write $\varphi_{\O_E}$ for the corresponding element of $\Omega$ constructed in Lemma \ref{lem:ramified_local_integrals} with $f = f_0$. Then Lemma \ref{lem:ramified_local_integrals} implies that it suffices to take 
\begin{gather*}
\varphi_0 \coloneqq \sum_{\substack{E\text{ with}\\ \epsilon(E, \chi, \psi) = \epsilon}}C_E^{-1}\varphi_{\O_E}. \qedhere
\end{gather*}
\end{proof}

\section{Archimedean test vectors and local integrals}\label{s:archimedean}
In this section, our goal is to prove Theorem \ref{thm:arch_integral}, which computes our local integrals at archimedean places. This calculation relies heavily on work of Pollack \cite{MR4094735}. First, in \S\ref{ss:freudenthal} and \S\ref{ss:lie} we recall some structural results about $\widetilde{G}(\RR)$ and its complexified Lie algebra. Then, in \S\ref{ss:Gtildearch} we define the element of $\Omega$ that we will use. Finally, in \S\ref{ss:archlocalintegral} we define the element of $\sigma^-$ that we will use and compute the associated integral.

\subsection{The Freudenthal construction of $\widetilde G$}\label{ss:freudenthal}
In the computations of this section, we will use the following alternate description of $\widetilde {\mathfrak g}$ from \cite[\S4]{MR4094735}. Write $\widetilde{\mathfrak m}^0 $ for the Lie algebra of $\widetilde{M}^1$. Then we have an identification
\begin{align*}
    \widetilde{\mathfrak g} = (\sl_2 \oplus \widetilde{\mathfrak m}^0 )\oplus (V_2 \otimes \widetilde{\X})
\end{align*}
\cite[Section 4.1, \S4.2.4]{MR4094735}, where $V_2$ denotes the standard representation of $\sl_2$ with standard basis $\set{e, f}$. Under this identification, the Lie bracket between the two summands is the natural action, and $\widetilde{\mathfrak{n}}$ corresponds to
\begin{align*}
    \big(\begin{bmatrix} 0 & * \\ 0 & 0 \end{bmatrix} \oplus 0\big)\oplus \big(e\otimes\widetilde{\X}\big) \subseteq \widetilde{\mathfrak g}.
\end{align*}
 
For the rest of this section, assume that $F$ is an archimedean local field. Choose $\widetilde{K}$ to be the maximal compact subgroup of $\widetilde{G}(\RR)$ with Lie algebra equal to the fixed points of the explicit Cartan involution from \cite[\S4.2.3]{MR4094735}. We identify $\widetilde{K}$ with $\SU(2)_\ell\times^{\{\pm1\}}\SU(6)/\mu_3(\C)$ using \cite[Proposition 4.1]{GW96}. By checking on Lie algebras, we see that
\begin{align*}
\SU(2)_\ell\times^{\{\pm1\}}\SU(2)_s\times\PU(3) \ra^\sim \widetilde{K}\cap(G(\RR)\times G'(\RR)),
\end{align*}
where 
the map $\SU(2)_\ell\to\widetilde{K}$ corresponds to the first factor, and the map $\SU(2)_s \times \PU(3) \to\widetilde{K}$ corresponds to the tensor product map $\SU(2) \times \SU(3) \to \SU(6)/\mu_3(\C)$ into the second factor.

In the notation of \cite[p.~1242]{MR4094735}, recall that the complexified Lie algebra of $\widetilde{K}$ is spanned by the following:
\begin{itemize}
    \item The complexified Lie algebra of $\SU(2)_\ell$ is given by the $\sl_2$-triple $\{e_\ell, f_\ell, h_\ell\}$ in $\widetilde{\g}_{\C}$. 
    \item The complexified Lie algebra of $\SU(6)/\mu_3(\C)$ is spanned by $n_E(Z)$, $n_H(Z)$, and $n_F(Z)$ for $Z$ in $J_\C$. 
\end{itemize}

\subsection{Some explicit elements in $\widetilde{\g}_\C$}\label{ss:lie}
In the notation of \cite{MR4094735}, we now describe some explicit elements in the complexified Lie algebra of $\widetilde{G}$. Write $J_0$ for the trace-zero subspace of $J$.
\begin{definition}
    For all $Z$ in $J_{0,\C}$,
    \begin{enumerate}
        \item Write $M(\Phi_{1,Z})$ in $\widetilde{\mathfrak m}^0_\C$ for the element acting on $\widetilde{\X}$ by
                $$M(\Phi_{1,Z}) (a, x, y, d) \coloneqq (0, 2Z\circ x, -2Z\circ y, 0).$$
        \item Write $n_L(Z)$ in $\widetilde{\mathfrak m}^0_\C$ for the element acting on $\widetilde{\X}$ by
$$n_L(Z) (a, x, y, d) \coloneqq (0, aZ, (x + Z)^\# - x^\# - Z^\#, \tr(y\circ Z)).$$
        \item Write $R(Z)$ in $\widetilde{\mathfrak{m}}^0_\C$ for $\frac12M(\Phi_{1,Z})+in_L(Z)$, and write $S(Z)$ in $\widetilde{\mathfrak{n}}_\C$ for $ie\otimes(0,iZ,-Z,0)$.
        \item Write $h_{-1}(Z)$ in $\widetilde{\g}_{\C}$ for $R(Z)+\frac12n_H(Z)$, and write $h_1(Z)$ in $\widetilde{\g}_{\C}$ for $S(Z)-n_F(Z)$.
    \end{enumerate}
\end{definition}
Since $J$ arises from the associative algebra $\operatorname{M}_3(\C)$, \cite[\S3.3.1]{MR4094735} shows that $M(\Phi_{1,Z})$ agrees with the notation from \cite[(3)]{MR4094735} and \cite[p.~1229]{MR4094735}. One immediately sees that $n_L(Z)$ agrees with the notation from \cite[p.~1228]{MR4094735}, and \cite[p.~1249]{MR4094735} shows that $h_{-1}(Z)$ and $h_1(Z)$ agree with the notation from \cite[p.~1242]{MR4094735}.

\begin{lemma}\label{lem:archcommutators}
     Let $Z_1$ and $Z_2$ be elements in $J_{0,\C}=\operatorname{M}_3(\C)^{\tr=0}$ satisfying
     \begin{align*}
         Z_1Z_2=Z_2Z_1=Z_1\circ Z_2=0\mbox{ and }(Z_1 + Z_2)^\# - Z_1^\# - Z_2^\# = 0.
     \end{align*}
     Then $[R(Z_1), n_H(Z_2)] = [R(Z_1), n_F(Z_2)] = [S(Z_1), n_H(Z_2)] = [S(Z_1), n_F(Z_2)] = 0$.
\end{lemma}
\begin{proof}
We claim that
\begin{align*}
[n_H(Z_1), n_H(Z_2)] = [n_H(Z_1), n_F(Z_2)] = [n_F(Z_1), n_F(Z_2)] = 0.
\end{align*}
To see this, use the description of $n_H$ and $n_F$ from \cite[p.~1246]{MR4094735} (where $\{v_1,v_2,v_3\}$ is defined in \cite[\S4.2.3]{MR4094735}) and apply \cite[Claim 6.3.2]{MR4094735} along with the last identity in \cite[\S3.3.1]{MR4094735}.  Using the claim, we can replace $S(Z_1)$ with $h_1(Z_1)$ and $R(Z_1)$ with $h_{-1}(Z_1)$. Then the lemma follows immediately from  \cite[Proposition 6.2.1]{MR4094735}.
\end{proof}

\subsection{Archimedean vectors for $\widetilde{G}$}\label{ss:Gtildearch}
Recall from \S\ref{ss:seesaw} that the sign $\epsilon$ equals $-1$, and recall from Definition \ref{defn:localsigma} the irreducible smooth representation $\sigma^-$ of $G'(\RR)$. Recall from \S\ref{ss:locint} the odd integer $N$ associated with~$\chi$, and write $m$ for the non-negative integer $\frac{\abs{N}-1}{2}$.

Recall that $\sigma^-$ has highest weight $(m, m, -2m)$ when $N$ is positive and $(2m, -m, -m)$ when $N$ is negative. To emphasize the dependence on $N$, we write $\sigma_N\coloneqq\sigma^-$. Write $\theta(-)$ for the theta lift from $G'$ to $G$ from \cite[Definition 2.5]{ourselves}, and recall that $\theta(\sigma_N)$ is isomorphic to the irreducible smooth representation $\pi_{m+1}$ of $G(\RR)$ from \S\ref{ss:qds} \cite[Theorem 5.2]{NDPC}.

We now define an element $\varphi_N$ of $\Omega$ by using certain raising operators. Define the matrices
\begin{align*}
Z_1 \coloneqq \begin{bmatrix}
    0&0& 1\\0&0&0 \\ 0&0&0 
\end{bmatrix}\quad\mbox{and}\quad Z_2 \coloneqq \begin{cases}\begin{bmatrix}
    0&0&0  \\0&0& 1\\ 0&0&0
\end{bmatrix} & \mbox{when }N\mbox{ is positive,} \\ \begin{bmatrix}
    0&1&0  \\0&0& 0\\ 0&0&0
\end{bmatrix} & \mbox{when }N\mbox{ is negative}
\end{cases}  
\end{align*}
in $J_{0,\C}=\operatorname{M}_3(\C)^{\tr=0}$, and define the raising operator
\begin{align*}
    \mathcal D_N \coloneqq \sum_{j = 0}^m (-1)^{j}\binom{m}{j} h_{-1}(Z_1)^{j}  h_{-1}(Z_2) ^{m-j} h_1(Z_1)^{m-j}  h_1(Z_2)^{j}  \in  (\widetilde{\g}_\C)^{\otimes2m}.
\end{align*}

Recall from \S\ref{ss:qds} that $\mathbb{V}_{m+1}\coloneqq\Sym^{2m+2}\boxtimes\one$ is the minimal $K$-type of $\pi_{m+1}$ and that $x_\ell^{2m+2}$ is a highest weight vector in $\mathbb{V}_{m+1}$, and recall from \S\ref{sss:localFourierE6arch} that $\widetilde{\mathbb{V}}_1\coloneqq\Sym^2\boxtimes\one$ is the minimal $\widetilde{K}$-type of $\Omega$ and that $x_\ell^2$ is a highest weight vector in $\widetilde{\mathbb{V}}_1$. Define $\varphi_N\coloneqq\mathcal D_Nx_\ell^2$ in $\Omega$.

\begin{prop}
\label{prop:hwv}
The image of $\varphi_N$ under the theta lift map $$\Omega\ra\sigma_N\boxtimes\theta(\sigma_N)\cong\sigma_N\boxtimes\pi_{m+1}$$ lies in $\sigma_N\boxtimes x_\ell^{2m+2}$.
\end{prop}
\begin{proof}
    By \cite[Proposition 4.2]{NDPC} and \cite[Theorem 5.2]{NDPC}, it suffices to show that $e_\ell\varphi_N = 0$ and $h_\ell\varphi_N = (2m+2)\phi_N$. Now \cite[p.~1251]{MR4094735} shows that $[e_\ell, h_1(Z)] = [e_\ell, h_{-1}(Z)] = 0$ for all $Z$ in $J_{0,\C}$, which implies that
    $$e_\ell\varphi_N = e_\ell\mathcal D_N x_\ell^2 = \mathcal D_N e_\ell x_\ell^2 = 0.$$
    Moreover, \cite[p.~1251]{MR4094735} also shows that
    $$[h_\ell, h_1(Z)] = h_1(Z)\mbox{ and }[h_\ell, h_{-1}(Z)] = h_{-1}(Z)$$ for all $Z$ in $J_{0,\C},$ which similarly implies that
    \begin{gather*}
        h_\ell\varphi_N = h_\ell \mathcal D_N x_\ell^2 = \mathcal D_N h_\ell x_\ell^2 + 2m \mathcal D_N x_\ell^2 = 2\mathcal{D}_Nx_\ell^2+2m\mathcal{D}_Nx_\ell^2 = (2m + 2) \varphi_N.\qedhere
    \end{gather*}
\end{proof}

\subsection{Archimedean local integrals}\label{ss:archlocalintegral}
Assume that there exists a $G'(\RR)$-orbit of $i$ in $\{E\hookrightarrow J\}$ satisfying $\Hom_{i(T_E)(\RR)}(\sigma^-,\one)\neq0$ (this $G'(\RR)$-orbit is unique by Proposition \ref{prop:epsilon_and_i}), and let $i$ be in this $G'(\RR)$-orbit. Because $G'(\RR)$ is compact, this forces $E\cong\RR^3$. Recall from Definition \ref{defn:localYang} the element $f_i$ of $\sigma^-$. Recall from \S\ref{sss:archYang} the Hermitian pairing $\langle-,-\rangle_\sigma$ on $\sigma_N$ and the fact that $\beta_i$ equals $\langle-,f_i\rangle_\sigma$.

Recall from \S\ref{ss:G2local} that $\mathcal{E}=(a,b,c,d)$ denotes an element of $\X(F)$ such that the associated cubic $F$-algebra is isomorphic to $E$. Write $\mathcal{X}=(a,x,y,d)$ for the element of $\O_{\min}(F)\cap p^{-1}(\mathcal{E})$ corresponding to $i$ under Lemma \ref{lem:integralminorbitfibers}, and recall from Definition \ref{defn:localzetaintegral} the integral $\mathcal{I}(\mathcal{E},-,-)$. 

Recall from \S\ref{ss:qds} the nonzero real number $r$. Write $(-,-)$ for the trace pairing $(X,Y)\mapsto\tr(X\circ Y)$ on $J$.

Write $T'\subseteq G'$ for the subgroup of diagonal matrices, and let $f_0$ be a highest weight vector in $\sigma_N$ with respect to $T'(\RR)$ that is unitary with respect to $\langle-,-\rangle_\sigma$. We now begin calculating our archimedean local integral:
\begin{prop}\label{prop:firststepforlocalzetaarch}
The integral $\mathcal{I}(\mathcal{E},\varphi_N,f_0)$ equals the product of 
\begin{align*}
\left[\frac{|r(ai+b-ci-d)|}{r(ai+b-ci-d)}\right]^{-m-1}K_{-m-1}(|r(ai+b-ci-d)|)
\end{align*}
and
$$\frac{(2ir^2)^m}2\int_{i(T_E)(\RR)\backslash G'(\RR)} \big[(g' \cdot Z_2, x) (g'\cdot Z_1, y) - (g'\cdot Z_2, y) (g'\cdot Z_1, x)\big]^m \overline{\langle g'\cdot f_0, f_i\rangle_\sigma}\, \dd{g'}.$$\end{prop}
\begin{proof}
It will be convenient to describe $\varphi_N$ using another differential operator instead. Write
$$\mathcal D_N' \coloneqq \sum_{j = 0}^m (-1)^j \binom{m}{j} R(Z_1)^jR(Z_2)^{m-j} S(Z_1)^{m-j} S(Z_2)^j\in  (\widetilde{\g}_\C)^{\otimes2m}.$$
Then the $\SU(6)$-invariance of $x_\ell^2$ and Lemma \ref{lem:archcommutators} indicate that $\mathcal D'_N x_\ell^2 = \mathcal D_N x_\ell^2= \varphi_N$, so Definition \ref{defn:archalpha} yields
\begin{align*}
\alpha_{\mathcal{X}}(\widetilde{m}\cdot\varphi_N) = \frac12(\mathcal D_N '\cdot\widetilde{\mathcal{W}}^{\mathcal{X}}_{-1})(\widetilde{m})
\end{align*}
for all $\widetilde{m}$ in $\widetilde{M}(\RR)$. Using the equivariance of $\widetilde{\mathcal{W}}^{\mathcal{X}}$ under left translation by $\widetilde{N}(\RR)$, for all $Z$ in $J_{0,\C}$, we get
\begin{align*}
    (S(Z)\widetilde{\mathcal W}^{\mathcal X}_{-1})(\widetilde{m}) = \langle r\mathcal X, \widetilde{m}\cdot (0, iZ, -Z, 0)\rangle\widetilde{\mathcal W}^{\mathcal X}_{-1}(\widetilde{m}).
\end{align*}
Because $\widetilde{N}/Z$ is abelian, for all $Z'$ in $J_{0,\C}$, the differential operator $S(Z)$ annihilates the function 
\begin{align*}
\widetilde{m}\mapsto\langle r\mathcal X, \widetilde{m}\cdot (0, iZ', -Z', 0)\rangle.
\end{align*}

Next, for all integers $v$, \cite[Corollary 7.6.1]{MR4094735}\footnote{A guide to the notation of \cite[Corollary 7.6.1]{MR4094735}: $D_{Z^*}(E)$ is defined on \cite[p.~1250]{MR4094735}, $M$ is defined in \cite[Theorem 7.5.1]{MR4094735}, and $V(E)^*$ is defined on \cite[p.~1242]{MR4094735}. We take $E=Z$ and $M=\widetilde{m}$. Then $D_{Z^*}(E)=R(Z)$ and $V(E)^*=(0,-iZ,-Z,0)$.} shows that
\begin{align*}
(R(Z)\widetilde{\mathcal{W}}^{\mathcal{X}}_{-v})(\widetilde{m})=\langle r\mathcal X, \widetilde{m}\cdot(0, -iZ, -Z,0)\rangle\widetilde{\mathcal{W}}^{\mathcal{X}}_{-v-1}(\widetilde{m})
\end{align*}
for all $\widetilde{m}$ in $\widetilde{M}^1(\RR)$. For all $k$ and $k'$ in $\{1,2\}$, the endomorphism of $\widetilde{N}/Z$ induced by $R(Z_k)$ annihilates $(0,\pm iZ_{k'},\pm Z_{k'},0)$, so the differential operator $R(Z_k)$ annihilates the function
\begin{align*}
\widetilde{m}\mapsto\langle r\mathcal{X},\widetilde{m}\cdot(0,\pm i Z_{k'},\pm Z_{k'},0)\rangle.
\end{align*}
Altogether, the above computations show that, for all $\widetilde{m}$ in $\widetilde{M}^1(\RR)$,
\begin{align*}
(R(Z_1)^jR(Z_2)^{m-j} S(Z_1)^{m-j} S(Z_2)^j\widetilde{W}_{-1}^{\mathcal{X}})(\widetilde{m})
\end{align*}
equals $\widetilde{W}^{\mathcal{X}}_{-m-1}(\widetilde{m})$ times
\begin{align*}
\langle r\mathcal X, \widetilde{m}\cdot(0, -iZ_1, -Z_1,0)\rangle^j\langle r\mathcal X, \widetilde{m}\cdot(0, -iZ_2, -Z_2,0)\rangle^{m-j}\langle r\mathcal X, \widetilde{m}\cdot (0, iZ_1, -Z_1, 0)\rangle^{m-j}\langle r\mathcal X, \widetilde{m}\cdot (0, iZ_2, -Z_2, 0)\rangle^j.
\end{align*}
Finally, specialize to $\widetilde{m}=g'$ in $G'(\RR)$. Since $g'$ lies in $\widetilde{K}\cap\widetilde{M}^1(\RR)$, it fixes $\widetilde{r}_0(i)$, so $\alpha_{\mathcal{X}}(g'\cdot\varphi_N)$ equals 
\begin{align*}
\widetilde{W}^{\mathcal{X}}_{-m-1}(g') = \left[\frac{|r(ai+b-ci-d)|}{r(ai+b-ci-d)}\right]^{-m-1}K_{-m-1}(|r(ai+b-ci-d)|)
\end{align*}
times
\begin{align*}
&\,\frac12\sum_{j=0}^m(-1)^j\binom{m}{j}r^{2m}(x-iy,g'\cdot Z_1)^j(x-iy,g'\cdot Z_2)^{m-j}(x+iy,g'\cdot Z_1)^{m-j}(x+iy,g'\cdot Z_2)^j \\
=&\,\frac{r^{2m}}2\big[(x-iy,g'\cdot Z_2)(x+iy,g'\cdot Z_1)-(x-iy,g'\cdot Z_1)(x+iy,g'\cdot Z_2)\big]^m\\
=&\,\frac{r^{2m}}2\big[-2i(y,g'\cdot Z_2)(x,g'\cdot Z_1)+2i(x,g'\cdot Z_2)(y,g'\cdot Z_1)\big]^m \\
=&\,\frac{(2ir^2)^m}2\big[(g' \cdot Z_2, x) (g'\cdot Z_1, y) - (g'\cdot Z_2, y) (g'\cdot Z_1, x)\big]^m.
\end{align*}
This immediately yields the desired result.
\end{proof}

To compute the integral from Proposition \ref{prop:firststepforlocalzetaarch}, we use two different models of $\sigma_N$:
\begin{enumerate}
    \item Write $\C^3$ for the standard representation of $\SU(3)$, write $(\C^3)^\vee$ for its dual, and write $\{z_1,z_2,z_3\}$ for the standard basis of $(\C^3)^\vee$. Checking highest weights yields an isomorphism of $\PU(3)$-representations $$\sigma_N \cong \begin{cases}\Sym^{3m}(\C^3)^\vee\otimes {\det}^{\otimes m} & \mbox{when }N\mbox{ is positive,} \\ \Sym^{3m}(\C^3) \otimes \det^{\otimes(-m)} & \mbox{when }N\mbox{ is negative.}\end{cases}$$
    
    \item Write $\lambda$ for the partition $(m,m)$ of $2m$, and write $c_\lambda$ in $\C[S_{2m}]$ for the associated Young symmetrizer  as in \cite[p.~46]{fultonharris1991}. Then the representation $\mathbb{S}^\lambda(\g'_\C)\coloneqq c_\lambda(\g'_\C)^{\otimes 2m}\subseteq(\g'_\C)^{\otimes2m}$ of $\PU(3)$ contains both $(m,m,-2m)$ and $(2m,-m,-m)$ as extremal weights with multiplicity one. Since $Z_1^{\otimes m}\otimes Z_2^{\otimes m}$ in $(\g'_\C)^{\otimes 2m}$ has weight $(m,m,-2m)$ when $N$ is positive and $(2m,-m,-m)$ when $N$ is negative, this implies that $\sigma_N$ is isomorphic to the subrepresentation of $\mathbb{S}^\lambda(\g'_\C)^{\otimes2m}$ generated by $c_\lambda(Z_1^{\otimes m}\otimes Z_2^{\otimes m})$.
\end{enumerate}
\begin{lemma}\label{lem:archmatrixcoeff}
For all $g'$ in $G'(\RR)$, we have $$\big[(g'\cdot Z_2, x) (g'\cdot Z_1, y) - (g'  \cdot Z_2, y)(g' \cdot Z_1,x)\big]^m = Cq(a, b, c, d)^{m/2} \langle g' \cdot f_0 , f_i\rangle_\sigma,$$
where $C$ is a nonzero constant independent of $\mathcal{E}$ and $\mathcal{X}$.
\end{lemma}
\begin{proof}
    First, we show that the desired identity holds up to \emph{some} constant, using model (2) for $\sigma_N$ above. Endow $\g'_\C=\operatorname{M}_3(\C)^{\tr=0}$ with the Hermitian pairing $(X,Y)\mapsto\text{tr}( X\cdot\prescript{t}{}{\overline Y} )$, which induces a Hermitian pairing $\langle-,-\rangle_\otimes$ on $(\g'_\C)^{\otimes2m}$ and hence on $\mathbb S^\lambda(\g'_\C)$, and write $\pr_\sigma$ in $\End_{\PU(3)}(\mathbb S^\lambda(\g'_\C))$ for the associated orthogonal projector onto $\sigma_N$. Because $\PU(3)$ preserves $\langle-,-\rangle_\otimes$, its restriction to $\sigma_N\subseteq\mathbb{S}^\lambda(\g'_\C)$ equals a scalar multiple of $\langle-,-\rangle_\sigma$. Moreover, $(x+iy)^{\otimes m}\otimes(x-iy)^{\otimes m}$ is fixed by $i(T_E)(\RR)$, the $i(T_E)(\RR)$-invariant subspace of $\sigma_N$ is $1$-dimensional, and $c_\lambda(Z_1^{\otimes m}\otimes Z_2^{\otimes m})$ is a highest weight vector in $\sigma_N$. This implies that there exists $C_1$ in $\C^\times$ such that, for all $g'$ in $G'(\RR)$, we have
    \begin{align*}
C_1 \langle g'\cdot f_0, f_i \rangle_\sigma &= 
    \langle \pr_\sigma c_\lambda (g'\cdot(Z_1^{\otimes m} \otimes Z_2^{\otimes m})), \pr_\sigma c_\lambda ((x+iy)^{\otimes m} \otimes (x-iy)^{\otimes m})\rangle_{\otimes}\\ &= \langle c_\lambda (g'\cdot( Z_1^{\otimes m} \otimes Z_2^{\otimes m})),c_\lambda((x+iy)^{\otimes m} \otimes (x-iy)^{\otimes m})\rangle_{\otimes}\\ &= 
\langle   c_\lambda ( (g'\cdot Z_1)^{\otimes m} \otimes (g'\cdot Z_2)^{\otimes m}),(x+iy)^{\otimes m} \otimes (x-iy)^{\otimes m}\rangle_{\otimes}.
\end{align*}
A combinatorial exercise using the definition of $c_\lambda$ shows that $c_\lambda ( (g'\cdot Z_1)^{\otimes m} \otimes (g'\cdot Z_2)^{\otimes m})$ equals
\begin{equation*}
    \sum_{j = 0 }^m (-1)^j \binom{m}{j}  \sum_{\sigma_1, \sigma_2\in S_m} \sigma_1\left((g'\cdot Z_1)^{\otimes j}\otimes (g'\cdot Z_2)^{\otimes (m-j)}\right) \otimes \sigma_2\left((g'\cdot Z_2)^{\otimes j} \otimes (g'\cdot Z_1)^{\otimes (m-j)}\right).
\end{equation*}
Combined with the $S_{2m}$-invariance of $\ang{-,-}_\otimes$, this shows that
\begin{align}
    C_1 \langle g'\cdot f_0 , f_i\rangle_\sigma &= \sum_{j = 0}^m (-1)^j \binom{m}{j} (m!)^2 (g'\cdot Z_1, x+iy)^j (g' \cdot Z_2, x+iy)^{m-j} (g'\cdot Z_2, x-iy)^j (g'\cdot Z_1, x-iy)^{m-j}\notag\\
    &= (m!)^2 \big[(g'\cdot Z_2, x +iy)(g' \cdot Z_1, x- iy)-(g'\cdot Z_1, x+iy)(g'\cdot Z_2, x - iy)  \big]^m\label{eq:C1}\\
    &= (-2i)^m (m!)^2 \big[(g'\cdot Z_2, x) (g'\cdot Z_1, y) - (g' \cdot Z_2, y)(g'\cdot Z_1, x)\big]^m.\notag
\end{align}
To finish the proof, we will compute $C_1$ by evaluating at a convenient point. For this, choose $h'$ in $G'(\RR)$ such that $h'\cdot x$ is of the form
$$\begin{bmatrix}
    x_1 && \\ & x_2&\\ &&x_3
\end{bmatrix},$$
where $x_1\geq x_2\geq x_3$. Because $h'\cdot\mathcal{X}=(a,h'\cdot x,h'\cdot y,d)$ and $f_{h'\cdot i}=h'\cdot f_i$, we can replace $x$ with $h'\cdot x$ to assume that $x$ is of the above form. Then $i(T_E)$ equals $T'$, and
$$y = x^\#/a = \begin{bmatrix}
    x_2x_3/ a && \\ &x_1x_3/a &\\&& x_1x_2 / a
\end{bmatrix}.$$

Since \eqref{eq:C1} is a polynomial identity in $g'$, it also holds for $g'$ in $G'(\C)$. Hence we can evaluate at
$$g' = 1+ s \prescript{t}{}{Z}_1 + t\prescript{t}{}{Z}_2 = \begin{cases}\begin{bmatrix}
    1 &  &  \\  & 1 &  \\ s & t & 1 
\end{bmatrix} & \mbox{when }N\mbox{ is positive,} \\ \begin{bmatrix}
    1 &  &  \\ t & 1 &  \\ s &  & 1
\end{bmatrix} & \mbox{when }N\mbox{ is negative,}\end{cases}$$ where $s$ and $t$ are complex numbers.

For the rest of the proof, assume that $N$ is positive; the other case is analogous. Then the right-hand side of \eqref{eq:C1} involves
\begin{align*}
    (g'\cdot Z_2, x)(g'\cdot Z_1, y) - (g'\cdot Z_2, y)(g'\cdot Z_1, x) &= \left( \begin{bmatrix}
         &  &  \\ -s & -t & 1 \\ -st & - t^2 &t
    \end{bmatrix}, x\right)\left( \begin{bmatrix}
        -s & - t & 1 \\  &  &  \\ -s^2 & -st & s
    \end{bmatrix}, y\right)\\&\quad - \left( \begin{bmatrix}
        -s & - t & 1 \\  &  &  \\ -s^2 & -st & s
    \end{bmatrix}, x\right)\left( \begin{bmatrix}
         &  &  \\ -s & -t & 1 \\ -st & - t^2 &t
    \end{bmatrix}, y\right) \\
    &=-ts(x_2 - x_3)(x_1 - x_3)x_2/a +st(x_1 - x_3) (x_2 - x_3)x_1/a \\
    &= st(x_2 - x_3)(x_1 - x_3)(x_1 - x_2)/a\\
    &= st q(a,b,c,d)^{1/2}.
\end{align*}
Finally, we calculate the left-hand side of \eqref{eq:C1} using model (1) for $\sigma_N$ above. Endow $\C^3$ with the standard pairing, which induces a Hermitian pairing $\langle-,-\rangle_{\Sym}$ on $\Sym^{3m}(\C^3)^\vee$. Then $z_3^{3m}$ is a unitary highest weight vector of $\Sym^{3m}(\C)^\vee\otimes\det^{\otimes m}$, and
\begin{align}\label{eq:archfixedvector}
\sqrt{\frac{(3m)!}{m!m!m!}} z_1^m z_2^m z_3^m 
\end{align}
is a unitary $i(T_E)(\RR)$-fixed vector of $\Sym^{3m}(\C)^\vee\otimes\det^{\otimes m}$.

Let $\sigma_N \cong \Sym^{3m} (\C^3)^\vee\otimes {\det}^{\otimes m}$ be the unique unitary isomorphism that sends $f_0$ to $z^{3m}_3$. Then the image of $f_i$ in $\Sym^{3m} (\C^3)^\vee\otimes {\det}^{\otimes m}$ is a $\C^1$-multiple of (\ref{eq:archfixedvector}), where the multiple does not depend on $\mathcal{X}$. After replacing $f_i$ with this $\C^1$-multiple, we get
\begin{align*}
    \langle g'\cdot f_0, f_i\rangle_\sigma &= \sqrt{\frac{(3m)!}{m!m!m!}} \langle g'\cdot z_3^{3m}, z_1^m z_2^mz_3^m \rangle_{\Sym} \\
    &= \sqrt{\frac{(3m)!}{m!m!m!}} \langle (-sz_1 - tz_2 + z_3)^{3m}, z_1^mz_2^m z_3^m \rangle_{\Sym} \\
    &= s^m t^m \sqrt{\frac{(3m)!}{m!m!m!}}^3 \langle z_1^m z_2^m z_3^m, z_1^mz_2^mz_3^m \rangle_{\Sym}\\
    &= s^m t^m \sqrt{\frac{(3m)!}{m!m!m!}}.
\end{align*}
We conclude by taking $s=t=1$ and comparing both sides of \eqref{eq:C1}.
\end{proof}
By putting everything together, we finally obtain:
\begin{thm}\label{thm:arch_integral}
    The integral $\mathcal{I}(\mathcal{E},\varphi_N,f_0)$ equals the product of 
\begin{align*}
\left[\frac{|r(ai+b-ci-d)|}{r(ai+b-ci-d)}\right]^{-m-1}K_{-m-1}(|r(ai+b-ci-d)|)
\end{align*}
and
$$C_N \cdot q(a,b,c,d)^{m/2},$$ where $C_N$ is a nonzero constant independent of $\mathcal{E}$ and $\mathcal{X}$. 
\end{thm}
\begin{proof}
    After combining Proposition \ref{prop:firststepforlocalzetaarch} and Lemma \ref{lem:archmatrixcoeff}, we conclude by observing that
    \begin{align*}
        \int _{i(T_E)(\RR)\backslash G'(\RR)} \langle g' \cdot f_0, f_i\rangle_\sigma \overline{\langle g' \cdot f_0, f_i\rangle_\sigma } \dd{g'} &= \dim(\sigma_N)^{-1}\vol(i(T_E)(\RR)\backslash G'(\RR))\langle f_0, f_0 \rangle_\sigma \overline{\langle f_i, f_i\rangle_\sigma}\\
        &= \dim(\sigma_N)^{-1}.\qedhere
    \end{align*}
\end{proof}
\bibliographystyle{amsplain}
\bibliography{bibliography}

\providecommand{\bysame}{\leavevmode\hbox to3em{\hrulefill}\thinspace}
\providecommand{\MR}{\relax\ifhmode\unskip\space\fi MR }
\providecommand{\MRhref}[2]{%
  \href{http://www.ams.org/mathscinet-getitem?mr=#1}{#2}
}
\providecommand{\href}[2]{#2}
\begin{thebibliography}{10}

\bibitem{Arthur}
James Arthur, \emph{Unipotent automorphic representations: conjectures}, no.
  171-172, 1989, Orbites unipotentes et repr\'{e}sentations, II, pp.~13--71.
  \MR{1021499}

\bibitem{ourselves}
Petar {Baki{\'c}}, Aleksander {Horawa}, Siyan~Daniel {Li-Huerta}, and Naomi
  {Sweeting}, \emph{{Global long root $A$-packets for $\mathsf{G}_2$: the
  dihedral case}}, arXiv e-prints (2024), arXiv:2405.17375.

\bibitem{bakic2021howe}
Petar Baki\'c and Gordan Savin, \emph{Howe duality for a quasi-split
  exceptional dual pair}, Math. Ann. \textbf{389} (2024), no.~1, 325--364.
  \MR{4735949}

\bibitem{Baruch_Mao}
Ehud~Moshe Baruch and Zhengyu Mao, \emph{Central value of automorphic
  {$L$}-functions}, Geom. Funct. Anal. \textbf{17} (2007), no.~2, 333--384.
  \MR{2322488}

\bibitem{Blasius}
Don Blasius, \emph{On the critical values of {H}ecke {$L$}-series}, Ann. of
  Math. (2) \textbf{124} (1986), no.~1, 23--63. \MR{847951}

\bibitem{DF40}
B.~N. Delone and D.~K. Faddeev, \emph{Theory of {I}rrationalities of {T}hird
  {D}egree}, Acad. Sci. URSS. Trav. Inst. Math. Stekloff, \textbf{11} (1940),
  340. \MR{4269}

\bibitem{fultonharris1991}
William Fulton and Joe Harris, \emph{Representation theory}, Graduate Texts in
  Mathematics, vol. 129, Springer-Verlag, New York, 1991, A first course,
  Readings in Mathematics. \MR{1153249}

\bibitem{GGS02}
Wee~Teck Gan, Benedict Gross, and Gordan Savin, \emph{Fourier coefficients of
  modular forms on {$G_2$}}, Duke Math. J. \textbf{115} (2002), no.~1,
  105--169. \MR{1932327}

\bibitem{Gan_Gurevich}
Wee~Teck Gan and Nadya Gurevich, \emph{C{AP} representations of {$G_2$} and the
  spin {$L$}-function of {${\rm PGSp}_6$}}, Israel J. Math. \textbf{170}
  (2009), 1--52. \MR{2506316}

\bibitem{gan2005minimal}
Wee~Teck Gan and Gordan Savin, \emph{On minimal representations definitions and
  properties}, Represent. Theory \textbf{9} (2005), 46--93. \MR{2123125}

\bibitem{gan2021twisted}
\bysame, \emph{Twisted composition algebras and {A}rthur packets for triality
  {$\rm Spin(8)$}}, Pure Appl. Math. Q. \textbf{18} (2022), no.~5, 1951--2130.
  \MR{4538042}

\bibitem{GRS93}
S.~Gelbart, J.~Rogawski, and D.~Soudry, \emph{On periods of cusp forms and
  algebraic cycles for {${\rm U}(3)$}}, Israel J. Math. \textbf{83} (1993),
  no.~1-2, 213--252. \MR{1239723}

\bibitem{GW96}
Benedict~H. Gross and Nolan~R. Wallach, \emph{On quaternionic discrete series
  representations, and their continuations}, J. Reine Angew. Math. \textbf{481}
  (1996), 73--123. \MR{1421947}

\bibitem{NDPC}
Jing-Song Huang, Pavle Pand\v{z}i\'{c}, and Gordan Savin, \emph{New dual pair
  correspondences}, Duke Math. J. \textbf{82} (1996), no.~2, 447--471.
  \MR{1387237}

\bibitem{Jiang_Rallis}
Dihua Jiang and Stephen Rallis, \emph{Fourier coefficients of {E}isenstein
  series of the exceptional group of type {$G_2$}}, Pacific J. Math.
  \textbf{181} (1997), no.~2, 281--314. \MR{1486533}

\bibitem{kazhdan2004minimal}
D.~Kazhdan and A.~Polishchuk, \emph{Minimal representations: spherical vectors
  and automorphic functionals}, Algebraic groups and arithmetic, Tata Inst.
  Fund. Res., Mumbai, 2004, pp.~127--198. \MR{2094111}

\bibitem{Kim_Yamauchi}
Henry~H. {Kim} and Takuya {Yamauchi}, \emph{{On the Fourier expansion of
  Gan-Gurevich lifts on the exceptional group of type $G_2$}}, arXiv e-prints
  (2024), arXiv:2411.16953.

\bibitem{KMRT}
Max-Albert Knus, Alexander Merkurjev, Markus Rost, and Jean-Pierre Tignol,
  \emph{The book of involutions}, American Mathematical Society Colloquium
  Publications, vol.~44, American Mathematical Society, Providence, RI, 1998,
  With a preface in French by J. Tits. \MR{1632779}

\bibitem{Kohnen_Zagier}
W.~Kohnen and D.~Zagier, \emph{Values of {$L$}-series of modular forms at the
  center of the critical strip}, Invent. Math. \textbf{64} (1981), no.~2,
  175--198. \MR{629468}

\bibitem{Kohnen}
Winfried Kohnen, \emph{Fourier coefficients of modular forms of half-integral
  weight}, Math. Ann. \textbf{271} (1985), no.~2, 237--268. \MR{783554}

\bibitem{kudla1984seesaw}
Stephen~S. Kudla, \emph{Seesaw dual reductive pairs}, Automorphic forms of
  several variables ({K}atata, 1983), Progr. Math., vol.~46, Birkh\"{a}user
  Boston, Boston, MA, 1984, pp.~244--268. \MR{763017}

\bibitem{Kud87}
\bysame, \emph{Splitting metaplectic covers of dual reductive pairs}, Israel J.
  Math. \textbf{87} (1994), no.~1-3, 361--401. \MR{1286835}

\bibitem{ChaoLi}
C.~Li, \emph{{Fourier coefficients of automorphic forms on $G_2$ and central
  values of twisted $L$-functions}}, Unpublished note, 2015.

\bibitem{magaard1997exceptional}
K.~Magaard and G.~Savin, \emph{Exceptional {$\Theta$}-correspondences. {I}},
  Compositio Math. \textbf{107} (1997), no.~1, 89--123. \MR{1457344}

\bibitem{Ono63}
Takashi Ono, \emph{On the {T}amagawa number of algebraic tori}, Ann. of Math.
  (2) \textbf{78} (1963), 47--73. \MR{156851}

\bibitem{MR4094735}
Aaron Pollack, \emph{The {F}ourier expansion of modular forms on quaternionic
  exceptional groups}, Duke Math. J. \textbf{169} (2020), no.~7, 1209--1280.
  \MR{4094735}

\bibitem{Pollack_G_2}
\bysame, \emph{{Exceptional theta functions and arithmeticity of modular forms
  on $G_2$}}, arXiv e-prints (2022), arXiv:2211.05280.

\bibitem{Prasanna_On_the_Fourier}
Kartik Prasanna, \emph{On the {F}ourier coefficients of modular forms of
  half-integral weight}, Forum Math. \textbf{22} (2010), no.~1, 153--177.
  \MR{2604368}

\bibitem{rogawski1992multiplicity}
Jonathan~D. Rogawski, \emph{The multiplicity formula for {$A$}-packets}, The
  zeta functions of {P}icard modular surfaces, Univ. Montr\'{e}al, Montreal,
  QC, 1992, pp.~395--419. \MR{1155235}

\bibitem{rumelhart1997minimal}
Karl~E. Rumelhart, \emph{Minimal representations of exceptional {$p$}-adic
  groups}, Representation Theory \textbf{1} (1997), 133--181.

\bibitem{Shi73}
Goro Shimura, \emph{On modular forms of half integral weight}, Ann. of Math.
  (2) \textbf{97} (1973), 440--481. \MR{332663}

\bibitem{Tunnell}
J.~B. Tunnell, \emph{A classical {D}iophantine problem and modular forms of
  weight {$3/2$}}, Invent. Math. \textbf{72} (1983), no.~2, 323--334.
  \MR{700775}

\bibitem{Waldspurger_Sur_les_coeff}
J.-L. Waldspurger, \emph{Sur les coefficients de {F}ourier des formes
  modulaires de poids demi-entier}, J. Math. Pures Appl. (9) \textbf{60}
  (1981), no.~4, 375--484. \MR{646366}

\bibitem{Xiong}
Wei Xiong, \emph{On certain {F}ourier coefficients of {E}isenstein series on
  {$G_2$}}, Pacific J. Math. \textbf{289} (2017), no.~1, 235--255. \MR{3652462}

\bibitem{Yan97}
Tonghai Yang, \emph{Theta liftings and {H}ecke {$L$}-functions}, J. Reine
  Angew. Math. \textbf{485} (1997), 25--53. \MR{1442188}

\end{thebibliography}
\end{document}